\documentclass[11pt,a4paper]{amsart}
\usepackage{amssymb,amsmath,latexsym}
\usepackage{graphicx}
\usepackage{bm,bbm}
\usepackage{color}
\usepackage{tikz}
\usepackage{ifthen}
\usetikzlibrary{snakes}
\usetikzlibrary{patterns}
\usepackage{algorithm}
\usepackage{setspace}
\usepackage[noend]{algpseudocode}
\usepackage{pgfplots}
\usepackage{enumitem}
\usepackage{array,blkarray}
\usepackage{bigdelim}
\usepackage{listings} 
\usepackage[hidelinks]{hyperref}
\usepackage{isotope}
\usepackage{stmaryrd}
\usepackage[]{placeins}
\usepackage[normalem]{ulem}
\usepackage[
mathb,
mathx,
]{mathabx}

\newcolumntype{L}[1]{>{\raggedright\let\newline\\\arraybackslash\hspace{0pt}}m{#1}}
\newcolumntype{C}[1]{>{\centering\let\newline\\\arraybackslash\hspace{0pt}}m{#1}}
\newcolumntype{R}[1]{>{\raggedleft\let\newline\\\arraybackslash\hspace{0pt}}m{#1}}

\newcommand{\calL}{\mathcal{L}}
\newcommand{\StiefelV}{\mathrm{St}(p,V)}
\newcommand{\Stiefelpn}{\mathrm{St}(p,n)}
\newcommand{\Sskew}{\mathcal{S}_{\rm skew}(p)}
\newcommand{\Ssym}{\mathcal{S}_{\rm sym}(p)}
\newcommand{\equivcl}[1]{\lcorners #1\rcorners} 
\newcommand{\im}{\mathrm{im}}
\newcommand{\cbf}{\bm{c}}
\definecolor{corrRed}{RGB}{18,124,175} 

\newtheorem{theorem}{Theorem}

\theoremstyle{definition}

\newtheorem{example}[theorem]{Example}

\theoremstyle{remark}
\newtheorem{remark}[theorem]{Remark}

\newtheorem{assumption}[theorem]{Assumption}
\numberwithin{theorem}{section}
\numberwithin{equation}{section}
\numberwithin{table}{section}
\numberwithin{figure}{section}

\makeatletter
\newsavebox{\@brx}
\newcommand{\llangle}[1][]{\savebox{\@brx}{\(\m@th{#1\langle}\)}%
	\mathopen{\copy\@brx\kern-0.5\wd\@brx\usebox{\@brx}}}
\newcommand{\rrangle}[1][]{\savebox{\@brx}{\(\m@th{#1\rangle}\)}%
	\mathclose{\copy\@brx\kern-0.5\wd\@brx\usebox{\@brx}}}
\makeatother
\definecolor{myBlue}{RGB}{113,104,238} 
\definecolor{myGreen}{RGB}{154,205,50} 
\definecolor{myGreen2}{RGB}{114,175,30} 
\definecolor{myRed}{RGB}{180,50,50}  
\definecolor{myOrange}{RGB}{225,92,22} 
\definecolor{lgray}{RGB}{200,200,200} 
\definecolor{llgray}{RGB}{155,155,155} 
\definecolor{mycolor1}{rgb}{0.00000,0.44700,0.74100}%
\definecolor{mycolor2}{rgb}{0.85000,0.32500,0.09800}%
\definecolor{mycolor3}{rgb}{0.92900,0.69400,0.12500}%
\definecolor{mycolor4}{rgb}{0.49400,0.18400,0.55600}%
\definecolor{mycolor5}{rgb}{0.46600,0.67400,0.18800}%
\definecolor{mycolor6}{rgb}{0.30100,0.74500,0.93300}%
\definecolor{mycolor7}{rgb}{0.63500,0.07800,0.18400}%
\newcommand\N{\mathbb N}
\newcommand\C{\mathbb C}

\newcommand\R{\mathbb R}

\DeclareMathOperator{\diag}{diag}
\DeclareMathOperator{\Diag}{Diag}

\DeclareMathOperator{\tr}{tr}
\DeclareMathOperator{\trace}{tr}

\DeclareMathOperator{\grad}{grad}
\DeclareMathOperator{\Hess}{Hess}
\DeclareMathOperator{\qf}{qf}

\DeclareMathOperator{\sym}{sym}

\DeclareMathOperator{\Drm}{D}

\newcommand{\GrassV}{\mathrm{Gr}(p,V)}
\newcommand{\Grasspn}{\mathrm{Gr}(p,n)}
\newcommand{\OrthGr}{\mathrm{O}(p)}

\newcommand\calA{\mathcal A}
\newcommand\calB{\mathcal B}

\newcommand\calE{\mathcal E}
\newcommand\calF{\mathcal F}

\newcommand\calH{\mathcal H}
\newcommand\calI{\mathcal I}
\newcommand\calJ{\mathcal J}

\newcommand\calP{\mathcal P}
\newcommand\calR{\mathcal R}

\newcommand\calV{\mathcal V}

\newcommand\fraki{\mathfrak i}
\newcommand\frakj{\mathfrak j}

\newcommand{\etabf}{{\bm \eta}}
\newcommand{\xibf}{{\bm \xi}}
\newcommand{\phibf}{{\bm \phi}}
\newcommand{\zetabf}{{\bm \zeta}}

\newcommand{\psibf}{{\bm \psi}}
\newcommand{\xbf}{\bm{x}}
\newcommand{\ybf}{\bm{y}}
\newcommand{\qbf}{\bm{q}}

\newcommand{\ubf}{\bm{u}}
\newcommand{\vbf}{\bm{v}}
\newcommand{\wbf}{\bm{w}}

\def\dx{\,\text{d}x}

\def\dy{\,\text{d}y}

\newcommand{\out}[2]{\llbracket{#1},{#2}\rrbracket_H}

\begin{document}
\title[Riemannian Newton methods for problems of Kohn--Sham type]{Riemannian Newton methods for energy\\ minimization problems of Kohn--Sham type}
\author[]{R.~Altmann$^{\dagger}$, D.~Peterseim$^{\ddagger}$, T.~Stykel$^{\ddagger}$}
\address{${}^{\dagger}$ Institute of Analysis and Numerics, Otto von Guericke University Magdeburg, Universit\"atsplatz 2, 39106 Magdeburg, Germany}
\address{${}^{\ddagger}$ Department of Mathematics \& Centre for Advanced Analytics and Predictive Sciences (CAAPS), University of Augsburg, Universit\"atsstr.~12a, 86159 Augsburg, Germany}
\email{robert.altmann@ovgu.de, \{daniel.peterseim, tatjana.stykel\}@uni-a.de}
\thanks{The work of Daniel Peterseim is part of a project that has received funding from the European Research Council (ERC) under the European Union's Horizon 2020 research and innovation programme (Grant agreement No.~865751 --  RandomMultiScales).}
\date{\today}
%
%
\begin{abstract}
This paper is devoted to the numerical solution of constrained energy mi\-ni\-mization problems arising in computational physics and chemistry such as the Gross--Pitaevskii and Kohn--Sham models. In particular, we introduce the Riemannian Newton methods on the infinite-dimensional Stiefel and Grassmann manifolds. We study the geometry of these two manifolds, its impact on the Newton algorithms, and present expressions of the Riemannian Hessians in the infinite-dimensional setting, which are suitable for variational spatial discretizations. A series of numerical experiments illustrates the performance of the methods and demonstrates its supremacy compared to other well-established schemes such as the self-consistent field iteration and gradient descent schemes.  
\end{abstract}
%
%
\maketitle
%
%
%
%
\section{Introduction}
The Kohn--Sham model~\cite{HohK64,KohS65,LeB05} is a prototypical example of a constrained energy minimization problem stated on the infinite-dimensional Stiefel manifold. This means that the sought-after minimizer is
a~$p$-tuple of $L^2$-ortho\-nor\-mal functions. Another well-known example is the Gross--Pitaevskii model for Bose--Einstein condensates of ultracold bosonic gases~\cite{LSY01,PS03}. Here, the special case~$p=1$ is of interest, where we seek a single (minimizing) function on the unit sphere in $L^2$, representing a unit mass constraint. 
Since these two applications are relevant for different communities, numerical methods are mostly considered separately. One aim of this paper is to give a~unified approach to solving energy minimization problems. More precisely, we introduce Riemannian Newton methods for minimizing energy functionals of Kohn--Sham type, which also includes the Gross--Pitaevskii model. In general, the here considered PDE problems require a special treatment in terms of sparsity and dimension-independent methods, which is not part of existing general optimization packages. 

The numerical solution of the Gross--Pitaevskii model has been studied extensively in recent years. The most common numerical techniques are iterative methods based on {\em Riemannian {\em(}conjugate{\em)} gradient  descent methods} or {\em discretized Riemannian gradient flows} in various metrics~\cite{GaP01,BaD04,BCL06,KaE10,RSS09,DanP17,HenP20,Zha22,CheLLZ23}. A~conceptually different approach is the \mbox{\em $J$-method}~\cite{JarKM14,AltHP21} with its inimitable sensitivity with regard to spectral shifts, allowing remarkable speed-ups in a~Rayleigh quotient iteration manner. 
Reformulating the minimization problem as an~eigenvalue problem with eigenvector nonlinearity -- also known as nonlinear eigenvector problem -- the {\em self-consistent field iteration} (SCF) can be employed; see~\cite{Roo51,Can00,DioC07}. This method involves the solution of a linear eigenvalue problem in each step and is strongly connected to the Newton method~\cite{JarU22,HenJ23}. Considering the extended nonlinear system including the normalization constraint also allows a direct application of Newton or Newton-type methods~\cite{BaoT03,COR09,DuL22}. For an extended review on numerical methods for the Gross--Pitaevskii model, we refer to~\cite{HenJ23}. 

Most of the above approaches (with appropriate adjustments) have been applied to the Kohn--Sham model as well. This includes the~{\em direct constrained minimization algorithm}~\cite{YanMW06,AloA09,SchRNB09} and the {\em energy-adapted gradient descent method} \cite{AltPS21} -- both based on Riemannian optimization -- as well as {\em gradient flow schemes}~\cite{DaiWZ20,HuWJ23}. Moreover, the SCF algorithm with different types of mixing  is very popular in the computational chemistry community; see, e.g., \cite{CanL00,Can01,LiuWWUY15,CanKL21,BaiLL22}. For a discretized and simplified Kohn--Sham model (without the external potential and the exchange-correlation energy), global convergence and local second-order convergence of an~inexact Riemannian Newton method on the Grassmann manifold has been shown in~\cite{ZhaoBJ15}. A~collection of software packages for density functional theory problems can be found in \cite{JiaZWW22}. 
  
In this paper, the point of origin is an energy functional defined on the infinite-dimen\-sio\-nal Stiefel manifold, which we introduce in Section~\ref{sect:model}. For a better understanding, we recall definitions and properties of the Stiefel manifold and corresponding retractions in Section~\ref{sect:Stiefel}. Moreover, we provide formulae for the Riemannian gradient and the Riemannian Hessian which are needed for the Newton iteration. Since the considered energy functional is invariant under orthogonal matrices, we also discuss the infinite-dimensional Grassmann manifold and examine a~connection of its tangent space to a~certain subspace of the tangent space of the Stiefel manifold. The resulting Newton algorithms are then subject of Section~\ref{sect:Newton}. In particular, we present an inexact Riemannian Newton method on the Grassmann manifold. In Section~\ref{sect:numerics}, we consider the two mentioned examples of the Gross--Pitaevskii and the Kohn--Sham model in more detail. For both applications, we derive the formulae including a spatial discretization and illustrate the supremacy of the inexact Newton approach compared to well-established methods such as the SCF iteration and gradient descent schemes. 
\smallskip

\textbf{Notation} 
The sets of $p\times p$ real symmetric and skew-symmetric matrices are denoted by $\Ssym$ and $\Sskew$, respectively. For $M\in\R^{p\times p}$, we write~$\sym M=\frac{1}{2}(M+M^T)$ for the symmetric part, and~$\trace M$ denotes the trace of $M$. Further, $I_p$ and $0_p$ denote the $p\times p$ identity and zero matrices, respectively. The expression~$\diag(M)$ defines the column vector consisting of the diagonal elements of $M\in\R^{n\times n}$ and~$\Diag(v)$ denotes the diagonal matrix with components of the vector $v\in\R^n$ on the diagonal.
%
%
\section{The Energy Functional and Nonlinear Eigenvector Problems}\label{sect:model}
For a~given spatial domain $\Omega\subseteq \R^d$, $d\le3$, we consider the Hilbert spaces $L^2(\Omega)$ and
 \mbox{$\tilde{V}\!\subseteq\! H^1(\Omega)$}. 
For $p\geq 1$, we further define the Hilbert spaces $V=\tilde{V}^p$ and $H=[L^2(\Omega)]^p$ of $p$-frames. Throughout this paper, we assume that~$V$ is dense in $H$ and that $V\subseteq H\subseteq V^*$ form a~Gelfand triple, where $V^*$ denotes the dual space of $V$. 

For $\vbf = (v_1, \dots, v_p), \wbf = (w_1, \dots, w_p) \in H$, we define the~dot product 
\[
\vbf\cdot\wbf = \sum_{j=1}^p v_j w_j.
\] 
On the pivot space $H$, we further introduce an~outer product
\begin{equation}\label{eq:outer}
	\out{\vbf}{\wbf} 
	= \begin{bmatrix}
		(v_1, w_1)_{L^2(\Omega)} & \dots & (v_1, w_p)_{L^2(\Omega)} \\
		\vdots & \ddots & \vdots \\ 
		(v_p, w_1)_{L^2(\Omega)} & \dots & (v_p, w_p)_{L^2(\Omega)}
	\end{bmatrix}
	\in \R^{p\times p}
\end{equation}
and an~inner product
\begin{equation}\label{eq:inner}
	(\vbf, \wbf)_H 
	= \sum_{j=1}^p (v_j, w_j)_{L^2(\Omega)} 
	= \trace\, \out{\vbf}{\wbf}.
\end{equation}
%
The inner product \eqref{eq:inner} induces the norm~$\|\vbf\|_H=\sqrt{(\vbf,\vbf)_H}$ on $H$. The canonical identification $\calI\colon V\to V^*$ is defined by 
\[
	\langle\,\calI\vbf,\wbf\rangle=(\vbf,\wbf)_H\qquad 
	\text{for all } \vbf,\wbf\in V,
\]
where $\langle\,\cdot\,,\cdot\,\rangle$ denotes the duality pairing on $V^*\times V$. This identification operator can also be written in the form~$\calI = \frakj^*\circ \fraki_H\circ \frakj$ with the trivial embedding~$\frakj\colon V\to H$ (the injective identity operator), the Riesz isomorphism~$\fraki_H\colon H\to H^*$, which reads $\fraki_H(u) = (u,\,\cdot\,)_H$, and the adjoint operator $\frakj^*\colon H^*\to V^*$ satisfying $\frakj^*(f)=f\circ \frakj$ for all $f\in H^*$. Since all these operators act componentwisely, we have $\calI(\vbf\Lambda) = \calI(\vbf)\Lambda$ for all $\vbf\in V$ and $\Lambda\in\R^{p\times p}$. 
Moreover, since $V$ is a dense subspace of $H$, so is $\frakj(V)$. Hence, $\frakj^*$ is injective and as the composition of injective operators, $\calI$ is also injective. As a result, $\calI$ has a~left inverse $\calJ\colon V^*\to V$ such that $\calJ\calI\vbf=\vbf$ for all $\vbf\in V$. 
%
%
\subsection{Energy and applications}\label{sect:model:energy}
For a $p$-frame~$\phibf\in V$, we consider the energy functional 
\begin{align}
	\calE(\phibf)
	&= \frac 12\, \int_{\Omega} \tr\bigl((\nabla \phibf(x))^T\nabla\phibf(x)\bigr) \dx  
	+ \int_{\Omega} \vartheta(x)\, \rho(\phibf(x)) \dx 
	+ \frac12\, \int_{\Omega} \varGamma(\rho(\phibf(x))) \dx
	\label{eq:energy}
\end{align}
with an external potential~$\vartheta$, the density function $\rho(\phibf) = \phibf\cdot \phibf$, 
and the smooth nonlinearity $\varGamma(\rho)$. Our aim is to minimize this energy functional on the \emph{infinite-dimensional Stiefel manifold of index $p$} given by 
\begin{equation}\label{eq:StiefelV}
	\StiefelV
	= \big\{ \phibf\in V\enskip :\enskip \out{\phibf}{\phibf} = I_p \big\}.
\end{equation}
In other words, we are interested in solving the constrained minimization problem 
\begin{equation}\label{eq:minSt}
	\min_{\phibf\in \StiefelV} \calE(\phibf). 
\end{equation}
A~state of lowest energy is called the {\em ground state}. Such states play an~important role in quantum-mechanical models  as they represent a most stable configuration of atoms and molecules. These models include two famous applications in computational physics and chemistry. 
\begin{example}[Gross--Pitaevskii model]
\label{exp:GPEVP}
For $p=1$ and $\varGamma(\rho) = \frac 12 \kappa \rho^2$ with $\kappa\in\mathbb{R}$, the energy functional takes the form
\begin{align}
	\calE_{\rm GP}(\phi)
	= \frac 12\, \int_{\Omega} \|\nabla \phi(x)\|^2 \dx
	+ \int_{\Omega} \vartheta(x)\, \phi(x)^2 \dx 
	+ \frac\kappa4\, \int_{\Omega} \phi(x)^4 \dx.
	\label{eq:energyGP}
\end{align}
This is the well-known Gross--Pitaevskii energy used in the modeling of Bose--Einstein condensates of ultracold bosonic gases~\cite{LSY01,PS03}. Here,  $\vartheta \in L^\infty(\Omega)$ is the magnetic trapping potential, $\phi\in H_0^1(\Omega)$ is the quantum state of the Bose--Einstein condensate, and~$\kappa$ characterizes the strength and the direction of particle interactions. 
\end{example}
\begin{example}[Kohn--Sham model]
\label{exp:KS}
The (non-local) nonlinearity  
\[
	\varGamma(\rho) 
    = \rho\, \int_{\Omega} \frac{\rho(\phibf(y))}{\|x-y\|} {\,\rm d}y 
    + 2\,\rho\,\epsilon_\text{xc}(\rho)
\]
yields the Kohn--Sham energy functional 
\begin{align}
	\calE_{\rm KS}(\phibf)
	&= \frac 12\, \sum_{j=1}^p \int_{\Omega} \|\nabla \phi_j(x)\|^2 \dx
	+ \int_{\Omega} \vartheta_\text{ion}(x)\, \rho(\phibf(x)) \,{\rm d} x \notag \\
	&\qquad+ \frac 12 \int_{\Omega}\int_{\Omega} \frac{\rho(\phibf(x))\, \rho(\phibf(y))}{\|x-y\|} {\,\rm d}y \dx
	+ \int_{\Omega}\rho(\phibf(x))\, \epsilon_\text{xc}(\rho(\phibf(x))) \dx,
	\label{eq:energyKS}
\end{align}
where $\phibf$ denotes a~wave function with $p$ components called single-particle orbitals and $\rho(\phibf)$ is the electronic charge density.  Moreover, $\vartheta_\text{ion}$ is the ionic potential, and $\epsilon_\text{xc}(\rho)$ is the exchange-correlation energy per particle in a~homogeneous electron gas of density~$\rho$. This model is based on the so-called \emph{density functional theory} \cite{HohK64}, which allows a~significant reduction of the degrees of freedom \cite{KohS65,LeB05,CanCM12}. The last integral in~\eqref{eq:energyKS} is a~local density approximation to the exchange-correlation energy obtained by using semi-empirically knowledge of the model~\cite{PerZ81}. In the Kohn--Sham model,  a~ground state corresponds to the low-energy wave function of the considered molecule and the orthogonality condition $\out{\phibf}{\phibf} = I_p$ means that there is no interaction between the electrons in different orbitals. 
\end{example}
At this point, it should be emphasized that, since the energy functional $\calE$ in  \eqref{eq:energy} is invariant under orthogonal transformations, i.e. $\calE(\phibf)=\calE(\phibf Q)$ for all orthogonal matrices $Q\in \R^{p\times p}$, the optimal solution to the minimization problem~\eqref{eq:minSt} is not unique. 
To overcome this difficulty, we will transfer this problem to the infinite-dimensional Grassmann manifold defined in Section~\ref{sec:Grassmann}.
%
%
\subsection{Connection to nonlinear eigenvector problems}\label{sect:model:NLEVP}
We observe that the directional derivative of $\calE$ from~\eqref{eq:energy} at $\phibf\in V$ along $\wbf\in V$ has the form
\[
	{\rm D}\calE(\phibf)[\wbf] 
	= a_\phibf(\phibf,\wbf),
\]
where 
\begin{align}
	a_\phibf(\vbf,\wbf)
	= \int_\Omega \tr\bigl((\nabla \vbf)^T\nabla\wbf\bigr) \dx 
	+ 2 \int_\Omega  \vartheta\, \vbf\cdot\wbf \dx 
	+ \int_\Omega  \gamma(\rho(\phibf))\, \vbf\cdot\wbf \dx 
	\label{eq:aphi}
\end{align}
with $\gamma(\rho)=\frac{\rm d}{{\rm d}\rho}\varGamma(\rho)$. One can see that for fixed $\phibf\in V$, $a_\phibf$ is a symmetric bilinear form on $V\times V$. Further note that~$a_\phibf$ exhibits a special structure, namely  
\begin{equation}\label{eq:aphi}
	a_\phibf(\vbf,\wbf)
	= \sum_{j=1}^p \tilde a_\phibf(v_j, w_j) 
\end{equation}
with a symmetric bilinear form~$\tilde a_\phibf\colon \tilde{V}\times \tilde{V}\to \R$ given by
\[
	\tilde a_\phibf(v, w) 
	= \int_\Omega (\nabla v)^T\nabla w \dx 
	+ 2 \int_\Omega \vartheta\, v w \dx 
	+ \int_\Omega \gamma(\rho(\phibf))\, v w \dx.
\]
Within this paper, we assume that~$\tilde a_\phibf$ is bounded and coercive on~$\tilde{V}\times \tilde{V}$. Obviously, the bilinear form $a_\phibf$ inherits these properties such that $a_{\phibf}$ is also bounded and coercive on~$V\times V$. 

Introducing the Lagrangian $\calL(\phibf,\Lambda)=\calE(\phibf)-\frac{1}{2}\trace\big(\Lambda^T(\out{\phibf}{\phibf} - I_p)\big)$ with a~Lagrange multiplier $\Lambda\in\Ssym$, the first-order necessary optimality conditions for the minimization problem \eqref{eq:minSt} yield the~nonlinear eigenvector problem (NLEVP)
\begin{subequations}
\label{eq:NLEVP}
\begin{align}
	a_{\phibf_*}(\phibf_*,\wbf)-(\phibf_*\, \Lambda_*, \wbf)_H  
	&= 0 \qquad \text{for all }\wbf\in V,\\
	\out{\phibf_*}{\phibf_*} - I_p 
	&= 0_p 
\end{align}
\end{subequations}
with unknown~$\phibf_*\in V$, which is referred to as  the eigenvector, and $\Lambda_*\in\Ssym$, whose eigenvalues are the lowest $p$ eigenenergies of the system.  
Yet another formulation of the NLEVP~\eqref{eq:NLEVP} follows from the special structure of the bilinear form~$a_\phibf$ in~\eqref{eq:aphi}: seek~$\phibf_*=(\phi_{*,1}, \dots, \phi_{*,p})\in \StiefelV$ and $p$~eigenvalues $\lambda_1,\dots,\lambda_p\in\R$ such that 
\begin{align}
	\label{eq:NLEVPweakComponents}
	\tilde a_{\phibf_*}(\phi_{*,j}, v) 
	= \lambda_j\, (\phi_{*,j}, v)_{L^2(\Omega)}
	\qquad\text{ for all } v\in \tilde V.
\end{align}

For fixed $\phibf\in V$, we introduce the~operator $\calA_\phibf \colon V\to V^*$, defined by   
\[
	\langle\calA_\phibf \,\vbf, \wbf\rangle 
	= a_\phibf(\vbf,\wbf) \qquad\text{for all } \vbf,\wbf\in V.
\]
Then the NLEVP \eqref{eq:NLEVP} can be written as
\begin{subequations}
	\label{eq:NLEVPop}
	\begin{align}
		\calA_{\phibf_*} \phibf_*-\calI(\phibf_* \Lambda_*) 
		&= \mathbf{0}^{*}, \label{eq:NLEVPop:a}\\
		\out{\phibf_*}{\phibf_*} - I_p 
		&=  0_p, \label{eq:NLEVPop:b}
	\end{align}
\end{subequations}
where $\mathbf{0}^{*}\in V^*$ is the zero functional. Using the left inverse $\calJ$ of $\calI$, we find that 
\begin{align}\label{eq:Lambda}
	\Lambda_*
	= \out{\phibf_*}{\phibf_*} \Lambda_*
	= \out{\phibf_*}{\phibf_*\, \Lambda_*}
	= \out{\phibf_*}{\calJ\!\calA_{\phibf_*} \,\phibf_*}.
\end{align}
\begin{remark}\label{rem:symmetryBracket}
Due to the symmetry of the bilinear form~$\tilde a_\phibf$, we conclude that 
\begin{align*}
	\big( \phi_i, (\calJ\!\calA_\phibf \,\phibf)_j \big)_{L^2(\Omega)}
	= \tilde a_{\phibf}(\phi_i, \phi_j)
	= \tilde a_{\phibf}(\phi_j, \phi_i)
	= \big( \phi_j, (\calJ\!\calA_\phibf \,\phibf)_i \big)_{L^2(\Omega)}, \quad i,j=1,\ldots,p.
\end{align*}
This means that $\out{\phibf}{\calJ\!\calA_\phibf \,\phibf}$ is symmetric for any $\phibf\in V$. 
\end{remark}
%
%
\section{The Infinite-dimensional Stiefel and Grassmann Manifolds}\label{sect:Stiefel}
In this section, we summarize definitions and properties of the infinite-dimensional Stiefel and Grassmann manifolds and their tangent spaces, which lay the foundation of the Riemannian optimization schemes in the upcoming section. 
%
%
\subsection{The Stiefel manifold}
We consider the infinite-dimensional Stiefel manifold $\StiefelV$ defined in~\eqref{eq:StiefelV}. It is an~embedded submanifold of the Hilbert space~$V$ and has co-dimension $p\,(p+1)/2$; see \cite{AltPS21}.  The {\em tangent space} of $\StiefelV$ at $\phibf\in \StiefelV$ is given by  
\begin{align}\label{eq:TangSpaceSt} 
	T_\phibf\,\StiefelV
	= \big\{ \etabf\in V\enskip :\enskip \out{\etabf}{\phibf} + \out{\phibf}{\etabf} = 0_p \big\}. 
\end{align}
The Riemannian structure of the Stiefel manifold $\StiefelV$ strongly depends on an underlying metric.
Within this paper, we equip $\StiefelV$ with the Hilbert metric given by
\begin{equation}\label{eq:Hmetric}
	g(\etabf,\zetabf)
	=(\etabf,\zetabf)_H
	=\trace\, \out{\etabf}{\zetabf},
	\qquad \etabf,\zetabf\in T_\phibf\,\StiefelV. 
\end{equation}
The {\em normal space} with respect to $g$ is then defined as 
\[	
	T_\phibf^{\perp} \,\StiefelV
	= \bigl\{ \xbf \in V\enskip :\enskip g(\xbf,\etabf)=0 \text{ for all } \etabf\in T_\phibf\,\StiefelV \bigr\}.
\]
It can also be represented as 
\begin{equation}\label{eq:normalS}
T_\phibf^{\perp} \,\StiefelV
= \bigl\{ \phibf S \in V\enskip :\enskip S\in \Ssym\bigr\}.
\end{equation}
Further, any $\ybf\in V$ can be decomposed as $\ybf = \calP_{\phibf}^{}(\ybf) + \calP_{\phibf}^\perp(\ybf)$, where
\begin{align}
\label{eq:def:HProj}	
	\calP_{\phibf}(\ybf)
	= \ybf - \phibf \sym \out{\phibf}{\ybf} \qquad\text{and}\qquad
	\calP_{\phibf}^{\perp}(\ybf)
	=\phibf \sym \out{\phibf}{\ybf}
\end{align}
are  the orthogonal projections onto the tangent and normal spaces, respectively. 

The \emph{Riemannian gradient} of a~smooth function $\,\calE\colon\StiefelV\to\R$ with respect to the Hilbert metric~$g$ is the unique element $\grad \calE(\phibf)\in T_\phibf\,\StiefelV$ satisfying the condition 
\[
	g(\grad \calE(\phibf), \etabf) = \Drm \overline{\calE}(\phibf)[\etabf] \qquad \text{ for all }\etabf\in T_\phibf\,\StiefelV,
\]
where $\overline{\calE}$ denotes a~smooth extension of $\calE$ around $\phibf$ in $V$ and $\Drm \overline{\calE}(\phibf)$ is the Fr\'echet derivative of $\,\overline{\calE}$ in $V$.

For the energy functional $\calE$ in~\eqref{eq:energy}, the Riemannian gradient at $\phibf\in \StiefelV$ with respect to $g$ can be determined by using the $L^2$-Sobolev gradient $\nabla\,\overline{\calE}(\phibf)\in V$ which is defined as the Riesz representation of $\Drm\overline{\calE}(\phibf)$ in the Hilbert space~$V$ with respect to the inner product $g$. Then, for all $\wbf\in V$, we have 
\[
	\langle\calA_\phibf \,\phibf, \wbf\rangle 
	= a_\phibf(\phibf,\wbf) 
	= \Drm\overline{\calE}(\phibf)[\wbf]
	= \big(\nabla \,\overline{\calE}(\phibf),\wbf \big)_H 
	= \big\langle\calI \,\nabla \,\overline{\calE}(\phibf), \wbf \big\rangle
\]
and, hence, $\nabla\,\overline{\calE}(\phibf) =\calJ\!\calA_\phibf\,\phibf$. Furthermore, for all $\etabf\in T_{\phibf}\,\StiefelV$, we obtain 
\[
\big(\grad \calE(\phibf), \etabf\big)_H 
= \Drm\overline{\calE}(\phibf)[\etabf]
= \big(\nabla \,\overline{\calE}(\phibf),\etabf \big)_H. 
\]
This implies that 
\begin{equation}\label{eq:gradE_H}
	\grad \calE(\phibf) 
	= \calP_{\phibf}\big(\nabla\, \overline{\calE}(\phibf)\big)
	= \calP_{\phibf}\big(\calJ\!\calA_\phibf\,\phibf\big)
	= \calJ\!\calA_\phibf\,\phibf-\phibf\, \out{\phibf}{\calJ\!\calA_\phibf\,\phibf}.
\end{equation}

The \emph{Riemannian Hessian} of $\,\calE$ at  $\phibf\in\StiefelV$ with respect to the metric~$g$, denoted by $\Hess\calE(\phibf)$, is a~linear mapping on the tangent space $T_\phibf\,\StiefelV$ into itself which is defined by 
\[
\Hess\calE(\phibf)[\etabf] = \nabla_{\etabf} \grad \overline{\calE}(\phibf) \qquad \text{for all }\etabf\in T_\phibf\,\StiefelV,
\]
where $\nabla_{\etabf}$ denotes the covariant derivative along $\etabf$ with respect to the connection $\nabla$, cf.~\cite[Sect.~5.3]{AbsiMS08} for the finite-dimensional case.

The following theorem provides two expressions for the Riemannian Hessian of $\calE$ in terms of the directional derivative of $\grad \calE(\phibf)$ and the $L^2$-Sobolev Hessian $\nabla^2\,\overline{\calE}(\phibf)$ of~$\,\overline{\calE}$, which is a~linear operator mapping $\vbf\in V$ onto the Riesz representation of~$\Drm^2\overline{\calE}(\phibf)[\vbf,\,\cdot\,]$ with respect to the inner product~$(\,\cdot\,,\cdot\,)_H$. 
\begin{theorem}\label{th:HessSt}
Let $\phibf\in\StiefelV$ and $\etabf\in T_{\phibf}\,\StiefelV$. Then the Riemannian Hessian of a smooth function $\calE\colon\StiefelV\to \mathbb{R}$ admits the expressions
\begin{align}
	\Hess\calE(\phibf)[\etabf] 
	& = \calP_{\phibf} \bigl(\Drm\grad\calE(\phibf)[\etabf]\bigr) \label{eq:HessSt1}\\
	& = \calP_{\phibf}\big( \nabla^2\,\overline{\calE}(\phibf)[\etabf]-\etabf\sym \out{\phibf}{\nabla\,\overline{\calE}(\phibf)}\big),\label{eq:HessSt2}
\end{align}
where $\nabla\,\overline{\calE}(\phibf)$ and $\nabla^2\,\overline{\calE}(\phibf)$ denote, respectively, the $L^2$-Sobolev gradient and the $L^2$-Sobolev Hessian of a~smooth extension $\,\overline{\calE}$ of $\,\calE$ around $\phibf$ in $V$.
\end{theorem}
\begin{proof} 
Since $\StiefelV$ is an embedded submanifold of the Hilbert space $V$, the expression~\eqref{eq:HessSt1} can be shown similarly to the finite-dimensional case \cite[Prop.~5.3.2]{AbsiMS08}. Indeed, for all $\etabf,\xibf\in T_\phibf\,\StiefelV$, we have
\begin{align*}
\big(\Hess\calE(\phibf)[\etabf],\xibf\big)_H 
& =  \Drm^2 \overline{\calE}(\phibf)[\etabf,\,\xibf\,] 
=  \Drm\!\big(\! \Drm\overline{\calE}(\phibf)[\xibf]\big)[\etabf]\\
& = \Drm\!\big( \!\grad \calE(\phibf),\xibf\big)_H[\etabf]
= \big(\!\Drm \grad \calE(\phibf)[\etabf],\xibf\big)_H.
\end{align*}
This immediately implies \eqref{eq:HessSt1}.  

In order to prove \eqref{eq:HessSt2}, we first compute the directional derivative 
\begin{equation}\label{eq:DgradE}
	\Drm\grad\calE(\phibf)[\etabf] 
	= \Drm\big(\calP_{\phibf}(\nabla\,\overline{\calE}(\phibf)\big)[\etabf]
	 = \calP_{\phibf}(\nabla^2\,\overline{\calE}(\phibf)[\etabf]) 
	 + \Drm \calP_{\phibf}[\etabf]\nabla\,\overline{\calE}(\phibf).
\end{equation}
%
%
Let $\cbf(t)\subset \StiefelV$ be a smooth curve defined on a~neighborhood of $t=0$ such that $\cbf(0)=\phibf$ and $\tfrac{{\rm d}}{{\rm d}t}\cbf(0)=\etabf$. Then for all $\ybf\in V$, we have
\begin{align*}
	\Drm \calP_{\phibf}[\etabf]\, \ybf 
	& = \lim_{t\to 0}\frac{1}{t} \big(\calP_{\cbf(t)}(\ybf)-\calP_{\phibf}(\ybf)\big) \\
	& = \lim_{t\to 0}\frac{1}{t} \big(\ybf-\cbf(t)\sym\out{\cbf(t)}{\ybf}-\ybf+\cbf(0)\sym\out{\cbf(0)}{\ybf}\big)\\
	& = -\lim_{t\to 0}\frac{1}{t} \big(\cbf(t)\sym\out{\cbf(t)-\cbf(0)}{\ybf}+(\cbf(t)-\cbf(0))\sym\out{\cbf(0)}{\ybf}\big)\\
	& = -\phibf \sym\out{\etabf}{\ybf}-\etabf\sym\out{\phibf}{\ybf}.
\end{align*}
Inserting \eqref{eq:DgradE} into \eqref{eq:HessSt1} and taking into account that 
\begin{align*}
	\calP_{\phibf}\big(\Drm \calP_{\phibf}[\etabf] \nabla\,\overline{\calE}(\phibf)\big)
 	&= -\calP_{\phibf}\big(\phibf \sym\out{\etabf}{\nabla\,\overline{\calE}(\phibf)}+\etabf\sym\out{\phibf}{\nabla\,\overline{\calE}(\phibf)}\big) \\ 
 	&= -\calP_{\phibf}\big(\etabf\sym\out{\phibf}{\nabla\,\overline{\calE}(\phibf)}\big), 
\end{align*}
we obtain \eqref{eq:HessSt2}.
\end{proof}
In order to derive a formula for the Riemannian Hessian of the energy functional~$\calE$ in~\eqref{eq:energy}, we first compute the second-order derivative
\begin{align*}
	\Drm^2\overline{\calE} (\phibf)[\vbf,\wbf] 
	&= \lim_{t\to 0}\ \frac{1}{t}\, \big\langle \calA_{\phibf+t\vbf} (\phibf+t\vbf)
	-\calA_\phibf\,\phibf,\wbf \big\rangle \\
	&= \lim_{t\to 0}\ \frac{1}{t}\,\bigg(
	\int_\Omega \Big(\tr\bigl((\nabla (\phibf+t\vbf))^T\nabla\wbf\bigr)-\tr\bigl((\nabla \phibf)^T\nabla\wbf\bigr)\Big) \dx  \\
	&\qquad\qquad\quad +2\int_\Omega  \vartheta\, \big((\phibf+t\vbf)\cdot\wbf -\phibf\cdot\wbf\big) \dx \\
	&\qquad\qquad\quad + \int_\Omega \big(\gamma(\rho(\phibf+t\vbf))\, (\phibf+t\vbf)\cdot \wbf-\gamma(\rho(\phibf))\, \phibf\cdot \wbf \big)  \dx \bigg)\\
	&= \int_\Omega \tr\bigl((\nabla \vbf)^T\nabla\wbf\bigr) \dx  
	+ 2\int_\Omega  \vartheta\, \vbf\cdot\wbf \dx \\
  	&\quad + \int_\Omega \gamma(\rho(\phibf))\, \vbf\cdot \wbf \dx  
	+ 2\int_{\Omega} \beta(\rho(\phibf)) (\phibf\cdot\vbf)\, (\phibf\cdot \wbf) \dx \nonumber \\
	&= \langle \calA_\phibf\, \vbf + \calB_\phibf\,\vbf,\wbf\rangle,
\end{align*}
where $\beta(\rho)=\frac{\rm d}{{\rm d}\rho}\gamma(\rho)$ and the operator $\calB_\phibf\colon V\to V^*$ has the form
\begin{equation}\label{eq:Bphi}
	\langle \calB_\phibf\,\vbf,\wbf\rangle 
	= 2\int_{\Omega} \beta(\rho(\phibf)) (\phibf\cdot\vbf)\, (\phibf\cdot \wbf) \dx.
\end{equation}
Hence, the $L^2$-Sobolev Hessian of  $\,\overline{\calE}$ is given by $\nabla^2\,\overline{\calE}(\phibf)[\vbf]=\calJ\!\calA_\phibf\, \vbf+  \calJ\calB_\phibf\,\vbf$ for all $\vbf\in V$.
By the definition of the orthogonal projection onto~$T_\phibf\,\StiefelV$ in~\eqref{eq:def:HProj}, we conclude that for~$\etabf\in T_\phibf\,\StiefelV$ the Riemannian Hessian of~$\calE$ is given by 
\begin{align}
	\Hess\calE(\phibf) [\etabf] 
	&= \calP_{\phibf} \big( \calJ\!\calA_\phibf\, \etabf+  \calJ\calB_\phibf\,\etabf-\etabf\,  \out{\phibf}{ \calJ\!\calA_\phibf\,\phibf}\big) \notag \\
	&= \calJ\!\calA_\phibf\, \etabf + \calJ\calB_\phibf\,\etabf 
	-\etabf\,  \out{\phibf}{\calJ\!\calA_\phibf\,\phibf}
	-\phibf\sym \out{\phibf}{\calJ\!\calA_\phibf\, \etabf} \label{eq:defHess} \\
	&\qquad -\phibf\sym \out{\phibf}{\calJ\calB_\phibf\, \etabf}
	+ \phibf\sym \big(\out{\phibf}{\etabf}\out{\phibf}{\calJ\!\calA_\phibf\, \phibf)}\big). \notag
\end{align}

Within optimization methods, we need to transfer data from the tangent space to the manifold to keep the iterations on the search space. For this purpose, we can use retractions defined as follows. Let $T\StiefelV$ 
be the tangent bundle to $\StiefelV$. A~smooth mapping $\mathcal{R}\colon T\StiefelV \to \StiefelV$ is called a~\emph{retraction} if for all $\phibf\in\StiefelV$, the restriction of $\calR$ to $T_\phibf\,\StiefelV$, denoted by $\calR_\phibf$, satisfies the following properties:
\begin{enumerate}
	\item[1)] $\calR_\phibf(\bm{0}_\phibf)=\phibf$, where $\bm{0}_\phibf$ denotes the origin of $T_\phibf\,\StiefelV$, and 
	\item[2)] $\tfrac{{\rm d}}{{\rm d}t} \calR_\phibf (t\etabf)\big|_{t=0}=\etabf$ for all $\etabf\in T_\phibf\,\StiefelV$.
\end{enumerate}
Retractions provide first-order approximations to the~exponential mapping on a~Riemannian manifold  and are often much easier to compute. A retraction $\calR$ on $\StiefelV$ is of \emph{second-order}, if it satisfies $\tfrac{{\rm d}^2}{{\rm d}t^2} \calR_\phibf (t\etabf)\big|_{t=0}\in T_\phibf^{\perp}\,\StiefelV$ for all $(\phibf,\etabf)\in T\StiefelV$.

In \cite{AltPS21}, several retractions on the Stiefel manifold $\StiefelV$ have been introduced. They can be considered as an~extension of the corresponding concepts on the matrix Stiefel manifold (see, e.g., \cite{AbsiM12,SatA19}) to the infinite-dimensional case. 

For $\vbf\in V$ with linearly independent components, we consider the $qR$~decomposition  $\vbf=\qbf R$, where 
$\qbf\in\StiefelV$ and $R\in\mathbb{R}^{p\times p}$ is upper triangular. Such a~decomposition exists and is unique if we additionally require that $R$ has positive diagonal elements. Then the {\em $qR$~decomposition based retraction} is defined as
$\calR^{qR}(\phibf,\etabf)=\qf(\phibf+\etabf)$, where $\qf(\phibf+\etabf)$ denotes the factor from $\StiefelV$ in the 
$qR$ decomposition of $\phibf+\etabf$. Such a~factor can be computed, e.g., by the modified Gram-Schmidt orthonormalization procedure on~$V$ presented in~\cite{AltPS21}. 

%
An alternative retraction can be defined by using the polar decomposition $\vbf=\ubf S$, where $\ubf\in\StiefelV$ and $S\in\mathbb{R}^{p\times p}$ is symmetric and positive definite. Assuming that the components of $\vbf$ are linearly independent, $S=\out{\vbf}{\vbf}^{1/2}$ and $\ubf=\vbf\,\out{\vbf}{\vbf}^{-1/2}$ are uniquely defined.
This leads to the {\em polar decomposition based retraction}
\[
	\calR^{\rm pol}(\phibf,\etabf) = (\phibf+\etabf)\out{\phibf+\etabf}{\phibf+\etabf}^{-1/2}, 
\]
which is of second order. Indeed, computing the second-order derivative of $\calR^{\rm pol}_\phibf (t\etabf)$
at $t=0$ and exploiting \eqref{eq:normalS}, we obtain that
\[
	\frac{{\rm d}^2}{{\rm d}t^2} \calR^{\rm pol}_\phibf (t\etabf)\Big|_{t=0} 
	= -\phibf\,\out{\etabf}{\etabf}\in T_{\phibf}^{\perp}\,\StiefelV.
\]
Note that second-order retractions are advantageous for second-order Riemannian optimization methods; see, e.g. \cite[Sect.~6.3]{AbsiMS08}. 
%
%
\subsection{The Grassmann manifold}\label{sec:Grassmann}
Let $\OrthGr$ be the orthogonal group of $\R^{p\times p}$. 
Following \cite{SchRNB09}, we define the infinite-dimensional {\em Grassmann manifold} as the quotient 
\[
	\GrassV 
	= \StiefelV/\OrthGr  
\] 
of the Stiefel manifold $\StiefelV$ with respect to the equivalence relation 
\[
	\phibf\sim\hat{\phibf} \qquad 
	\Longleftrightarrow\qquad \hat{\phibf} = \phibf\, Q \text{ for some } Q\in \OrthGr.
\]
The Grassmann manifold $\GrassV$ can be interpreted as the set of the equivalence classes given by 
\[
\equivcl{\phibf} 
= \big\{ \hat{\phibf}\in\StiefelV\enskip :\enskip \hat{\phibf}=\phibf\, Q, \; Q\in\OrthGr \big\}
\] 
for $\phibf\in\StiefelV$. Similarly to the Grassmann matrix manifold \cite[Prop.~3.4.6]{AbsiMS08}, one can show that $\GrassV$ admits a~unique structure of quotient manifold. A {\em canonical projection} from the Stiefel manifold into the Grassmann manifold is defined by
\[
\arraycolsep=2pt
\begin{array}{rcl}
	\pi\colon\StiefelV & \to & \GrassV\\
	\phibf &\mapsto & \equivcl{\phibf}
\end{array}
\] 
and is a smooth submersion. This means that $\Drm\! \pi(\phibf)$ is surjective, and,  hence, the equivalence class $\pi^{-1}(\equivcl{\phibf})$  is an~embedded submanifold of $\StiefelV$; see \cite[Prop.~3.4.4.]{AbsiMS08}. 

In the following, we examine a useful connection of the Stiefel manifold and the Grassmann manifold. More precisely, we show that there is a one-to-one relation between the tangent space of the Grassmann manifold and the so-called horizontal space, a subspace of the tangent space of the Stiefel manifold. The tangent space $T_\phibf \, \StiefelV$ at $\phibf\in \StiefelV$ defined in~\eqref{eq:TangSpaceSt} can be splitted with respect to the projection $\pi$ and the Hilbert metric $g$ as $T_\phibf \, \StiefelV = \calV_{\phibf} \oplus \calH_\phibf$, where 
\begin{equation}\label{eq:vertV}
	\calV_{\phibf} 
	= T_\phibf \,\pi^{-1}(\equivcl{\phibf}) 
	= \big\{\phibf\, \varTheta\enskip :\enskip \varTheta\in \Sskew \big\}
\end{equation}
is the {\em vertical space} at $\phibf$ and 
\begin{align*}
	\calH_{\phibf} 
	= \calV_\phibf^\perp 
	&= \big\{ \xbf \in T_\phibf\, \StiefelV\enskip :\enskip g(\xbf,\vbf)=0 \text{ for all } \vbf\in \calV_\phibf\big\} \\
	&=\big\{ \xbf \in T_\phibf\, \StiefelV\enskip :\enskip \out{\phibf}{\xbf}=0_p\big\}
\end{align*}
is the {\em horizontal space} at $\phibf$; see \cite[Lem.~2]{SchRNB09}. The orthogonal projection of a tangent vector $\etabf\in T_\phibf \StiefelV$ onto $\calH_{\phibf}$ is given by 
\begin{align}
\label{eq:def:horizontalProj}
	\calP_{\phibf}^{\rm h}(\etabf) 
	= \etabf -\phibf\, \out{\phibf}{\etabf}.
\end{align}
One can see that, moving on a~curve in the Stiefel manifold $\StiefelV$ with direction in the vertical space $\calV_{\phibf}$, we stay in the equivalence class $\equivcl{\phibf}$. The tangent space $T_{\equivcl{\phibf}}\GrassV$ of the Grassmann manifold $\GrassV$ can then be identified with the horizontal space $\calH_{\phibf}$ in the sense that for any $\psibf\in T_{\equivcl{\phibf}}\GrassV$ there exists a~unique $\psibf_{\phibf}^{\rm h}\in \calH_{\phibf}$ such that \mbox{$\Drm\!\pi(\phibf)[\psibf_{\phibf}^{\rm h}]=\psibf$}. The unique element  $\psibf_{\phibf}^{\rm h}$ is called the {\em horizontal lift} of $\psibf$ at $\phibf$. This relation allows us to introduce a~metric on the Grassmann manifold $\GrassV$, namely 
\[
	g^{\rm Gr}(\psibf,\zetabf) 
	= g(\psibf^{\rm h},\zetabf^{\rm h}), \qquad 
	\psibf,\zetabf\in T_{\equivcl{\phibf}}\GrassV,\; \equivcl{\phibf}\in\GrassV,
\]
where $\psibf_{\phibf}^{\rm h}, \zetabf_{\phibf}^{\rm h}\in \calH_{\phibf}$ are the horizontal lifts of $\psibf$ and $\zetabf$ at $\phibf$, respectively. Due to $\psibf_{\phibf Q}^{\rm h}=\psibf_{\phibf}^{\rm h}Q$ for all $Q\in\OrthGr$, one can show that this metric does not depend on the choice of the representative~$\phibf$ of the equivalence class~$\equivcl{\phibf}$. 

The connection of $T_{\equivcl{\phibf}}\GrassV$ and $\calH_{\phibf}$ makes it possible to introduce optimization methods on the Grassmann manifold, while still working on the tangent space of the corresponding Stiefel manifold. 
Using the canonical projection $\pi$, the  minimization problem~\eqref{eq:minSt} on the Stiefel manifold $\StiefelV$ can be written as the minimization problem 
\begin{equation} 
	\label{eq:minGrV}
	\min_{\equivcl{\phibf}\in\GrassV} {\calF}(\equivcl{\phibf})
\end{equation}
on the Grassmann manifold $\GrassV$, where the cost functional $\calF\colon \GrassV\to \R$ is induced by $\calE$ as $\calE(\phibf)={\calF}(\pi(\phibf))$ and $\calF(\equivcl{\phibf})={\calE}(\pi^{-1}(\equivcl{\phibf}))$. Note that this definition is justified by the fact that $\calE(\phibf)=\calE(\phibf Q)$ for all~$Q\in\OrthGr$. The horizontal lift of the Riemannian gradient $\grad {\calF}(\equivcl{\phibf})\in T_{\equivcl{\phibf}}\GrassV$ with respect to the metric $g^{\rm Gr}$ is given by 
\begin{equation}\label{eq:gradFGr}
\grad {\calF}(\equivcl{\phibf})_{\phibf}^{\rm h} 
= \grad{\calE}(\phibf) 
= \calP_{\phibf}^{\rm h} \big(\calJ\!\calA_\phibf\,\phibf \big)
= \calJ\!\calA_\phibf\,\phibf - \phibf\, \out{\phibf}{\calJ\!\calA_\phibf\,\phibf}.
\end{equation}
To obtain the horizontal lift of the Riemannian Hessian $\Hess {\calF}(\equivcl{\phibf})[\psibf]$, we proceed as before but replace the projection~$\calP_{\phibf}$ by the orthogonal projection $\calP_{\phibf}^{\rm h}$ onto the horizontal space; see equation~\eqref{eq:def:horizontalProj}. This leads to 
\begin{align}
	(\Hess{\calF}(\equivcl{\phibf})[\psibf])_{\phibf}^{\rm h} 
	&= \calP_{\phibf}^{\rm h} \big(\Drm\grad \calE(\phibf)[\psibf_{\phibf}^{\rm h}] \big) \notag \\
	&= \calP_{\phibf}^{\rm h} \big(\calJ\!\calA_\phibf\, \psibf_{\phibf}^{\rm h}+ \calJ\calB_\phibf\,\psibf_{\phibf}^{\rm h}-\psibf_{\phibf}^{\rm h}\out{\phibf}{\calJ\!\calA_\phibf\,\phibf} \big).
	\label{eq:HessGrassmann}
\end{align}

Retractions on the Grassmann manifold are inherited from that on the Stiefel manifold applied to the horizontal lift; see \cite[Prop.~4.1.3]{AbsiMS08}. For all \mbox{$\equivcl{\phibf}\in\GrassV$} and $\psibf\in T_{\equivcl{\phibf}}\GrassV$, we have
\begin{align*}
	\calR^{{\rm Gr,pol}}(\equivcl{\phibf}, \psibf) & = \pi\big(\calR^{\rm pol}(\phibf+\psibf_{\phibf}^{\rm h})\big), \\
	\calR^{{\rm Gr,qR}}(\equivcl{\phibf}, \psibf) & = \pi\big(\calR^{\rm qR}(\phibf+\psibf_{\phibf}^{\rm h})\big).
\end{align*}
Note that these retractions are independent of the chosen point $\phibf$, providing the same equivalence class on $\GrassV$.

Similar to the matrix case \cite{AbsiMS08,EdeAS98}, we can also derive an explicit expression for the \emph{Grassmann exponential} ${\rm Exp}\colon T\GrassV\to \GrassV$, which maps \mbox{$(\equivcl{\phibf},\psibf)\in  T\GrassV$} to the end point of the unique geodesic starting at $\equivcl{\phibf}$ and going in the direction~$\psibf$. Let $\psibf_{\phibf}^{\rm h}=\ubf\Sigma W^T$ be a singular value decomposition of the horizontal lift $\psibf_{\phibf}^{\rm h}$ of~$\psibf$, where $\ubf\in\StiefelV$, $W\in\OrthGr$, and $\Sigma\in\R^{p\times p}$ is diagonal with nonnegative diagonal elements. Then the Grassmann exponential is given by 
\[
{\rm Exp}(\equivcl{\phibf}, \psibf) = \pi \big(\phibf \,W\cos \Sigma +\ubf \sin\Sigma\big).
\]
Using the property $\psibf_{\phibf}^{\rm h}\in\calH_\phibf$, one can verify that $\phibf \,W\cos \Sigma +\ubf \sin\Sigma \in\StiefelV$. Therefore, it can be considered as a representative of the resulting equivalence class.
\begin{remark}
	For $p=1$, the Stiefel manifold coincides with the Grassmann manifold and equals the unit sphere
	\[
	\mathbb{S}= \big\{ \phi\in \tilde{V}\enskip :\enskip \|\phi\|_{L^2(\Omega)} = 1 \big\}.
	\]
	Its tangent space is given by $T_{\phi}\,\mathbb{S}	= \{ \eta\in \tilde{V} : (\eta, \phi)_{L^2(\Omega)} = 0 \}$ and the orthogonal projection onto this space takes the form $\calP_{\phi}(y)=y-(\phi,y)_{L^2(\Omega)}\phi$ for $y\in \tilde{V}$. 
	Furthermore, for $(\phi,\eta)\in T\mathbb{S}$, the second-order retraction and the exponential mapping on $\mathbb{S}$ are given by  
	\begin{align*}
	\calR(\phi,\eta) =\frac{\phi+\eta\qquad}{\|\phi+\eta\|_{L^2(\Omega)}}, \qquad 
	{\rm Exp}(\phi,\eta)=\cos(\|\eta\|_{L^2(\Omega)})\phi+\sin(\|\eta\|_{L^2(\Omega)})\frac{\eta\qquad}{\|\eta\|_{L^2(\Omega)}},
\end{align*}
respectively.
\end{remark}
%
%
\section{Riemannian Newton Methods}\label{sect:Newton}
In this section, we present Riemannian Newton methods on the Stiefel as well as on the Grassmann manifold and discuss the inexact version of the latter.

Within the Riemannian Newton method on the Stiefel manifold $\StiefelV$, for given $\phibf_k\in \StiefelV$, we first compute the Newton search direction~$\etabf_k\in T_{\phibf_k} \StiefelV$ by solving the Newton equation 
\begin{equation}\label{eq:NewtonEqSt}
	\Hess\calE(\phibf_k) [\etabf_k]
	= -\grad \calE(\phibf_k).
\end{equation}
The iterate is then updated by~$\phibf_{k+1}=\calR(\phibf_k,\etabf_k)$ for any retraction on $\StiefelV$. 
It should, however, be noted that due to the non-uniqueness of the minimizer of \eqref{eq:minSt} caused by the invariance of $\calE$ under orthogonal transformations, we cannot expect that $\Hess \calE(\phibf_k)$ is invertible on $T_{\phibf_k} \StiefelV$, which implies that the solution of~\eqref{eq:NewtonEqSt} is non-unique. This is, indeed, vindicated by the following theorem.
\begin{theorem}\label{th:HessStiefel}
Let $\phibf\in\StiefelV$ and let $\calV_{\phibf}$ be the vertical space  given in \eqref{eq:vertV}. Then the Riemannian Hessian of the energy functional~$\calE$ from~\eqref{eq:energy} in~$\phibf$ is non-invertible on~$\calV_{\phibf}$, i.e., $\big(\xibf,\Hess\calE(\phibf) [\etabf]\big)_H=0$ for all $\etabf,\xibf\in\calV_{\phibf}$. 
\end{theorem}
\begin{proof}
Let $\etabf,\xibf\in\calV_{\phibf}$ be arbitrary. Then there exist $\varTheta_{\etabf},\varTheta_{\xibf}\in\Sskew$ such that \mbox{$\etabf=\phibf\,\varTheta_{\etabf}$} and $\xibf=\phibf\,\varTheta_{\xibf}$. Using the definition of $\Hess\calE(\phibf)$ in~\eqref{eq:defHess} and the symmetry of the matrix~$\out{\phibf}{\calJ\!\calA_\phibf\, \phibf}$ shown in Remark~\ref{rem:symmetryBracket}, we have
\begin{align*}
	\Hess\calE(\phibf) [\etabf] 
	&= \calJ\!\calA_\phibf\, \phibf\,\varTheta_{\etabf} 
	+ \calJ\calB_\phibf\,\phibf\,\varTheta_{\etabf} 
	- \phibf\,\varTheta_{\etabf}\,  \out{\phibf}{\calJ\!\calA_\phibf\,\phibf} 
	- \phibf\sym\big( \out{\phibf}{\calJ\!\calA_\phibf\, \phibf} \varTheta_{\etabf}\big)\\
	&\qquad -\phibf\sym \big(\out{\phibf}{\calJ\calB_\phibf\, \phibf}\varTheta_{\etabf}\big)
	+ \phibf\sym \big(\varTheta_{\etabf}\out{\phibf}{\calJ\!\calA_\phibf\, \phibf}\big)\\
	& = \calJ\!\calA_\phibf\, \phibf\,\varTheta_{\etabf} 
	-\phibf\,  \out{\phibf}{\calJ\!\calA_\phibf\,\phibf} \,\varTheta_{\etabf} +\calJ\calB_\phibf\,\phibf\,\varTheta_{\etabf}
	-\phibf\sym \big(\out{\phibf}{\calJ\calB_\phibf\, \phibf}\varTheta_{\etabf}\big) 
\end{align*}
and, hence,
\begin{align*}
	\big(\xibf,\Hess\calE(\phibf) [\etabf]\big)_H  
	&=\trace\left(\varTheta_{\xibf}^T\out{\phibf}{\calJ\calB_\phibf\,\phibf} \,\varTheta_{\etabf}^{}-\Theta_{\xibf}^T\sym \big(\out{\phibf}{\calJ\calB_\phibf\, \phibf}\varTheta_{\etabf}^{}\big)\right).
\end{align*}
We now show that $\out{\phibf}{\calJ\calB_\phibf\,\phibf} \,\varTheta_{\etabf}= 0$. With the skew-symmetric matrix $\varTheta_{\etabf} = [\theta_{ij}]_{i,j=1}^p$, we first observe that 
\[
\phibf\cdot(\phibf\,\varTheta_{\etabf}) 
= \sum_{j=1}^p \phi_j \sum_{i=1}^p \phi_i\theta_{ij}
= \sum_{i=1}^p \phi_i \sum_{j=1}^p \phi_j\theta_{ij}
= -\sum_{i=1}^p \phi_i \sum_{j=1}^p \phi_j\theta_{ji}
= -\phibf\cdot(\phibf\,\varTheta_{\etabf}). 
\]
This implies $\phibf\cdot(\phibf\,\varTheta_{\etabf}) = 0$ and, hence, 
$\out{\phibf}{\calJ\calB_\phibf\,\phibf} \,\varTheta_{\etabf}= 0$. As a result, we conclude that $\big(\xibf,\Hess\calE(\phibf) [\etabf]\big)_H=0$ for all $\etabf,\xibf\in\calV_{\phibf}$. 
\end{proof} 
It follows from Theorem~\ref{th:HessStiefel} that if equation \eqref{eq:NewtonEqSt} is solvable, its solution is non-unique. To overcome this difficulty, we pass
on to the Grassmann manifold $\GrassV$. Given $\equivcl{\phibf_k}\in\GrassV$, the Newton direction
$\psibf_k\in T_{\equivcl{\phibf_k}}\GrassV$ is computed by solving the Newton equation
\[
	\Hess\calF(\equivcl{\phibf_k}) [\psibf_k]
= -\grad \calF(\equivcl{\phibf_k}).
\]
By applying the horizontal lift expressions \eqref{eq:gradFGr} and \eqref{eq:HessGrassmann}, this equation leads to 
 \begin{equation}
\label{eq:NewGrV}
	\calP_{\phibf_k}^{\rm h} \big(\Drm\grad \calE(\phibf_k)[(\psibf_k)_{\phibf_k}^{\rm h}] \big) 
	= -\grad \calE(\phibf_k)
\end{equation}
with unknown $(\psibf_k)_{\phibf_k}^{\rm h}\in\calH_{\phibf_k}$ being the horizontal lift of~$\psibf_k$ at~$\phibf_k$. Note that this equation is well-defined, since~$-\grad \calE(\phibf_k)$ is an element of the horizontal space $\calH_{\phibf_k}$. 

For solving the Newton equation \eqref{eq:NewGrV}, we can employ any matrix-free iterative linear solver which does not require the storage of the coefficient matrix explicitly but accesses it by computing the matrix-vector product or -- as in our case --  by evaluating the linear operator given in~\eqref{eq:HessGrassmann}. The resulting Riemannian Newton method on the Grassmann manifold is presented
in Algorithm~\ref{alg:infdim:RiemNewtonGr}.
\begin{algorithm}[h]
	\setstretch{1.15}
	\caption{Riemannian Newton method on the Grassmann manifold} 
	\label{alg:infdim:RiemNewtonGr}
	\begin{algorithmic}[1]
		\State {\bf Input}: initial guess $\phibf_0\in\StiefelV$, parameters $\delta,\eta\in(0,1)$, $\sigma\in(0,1/2]$, $\ell_{\max}\in\N$
		\For{$k=0,1,2\ldots$} 
		\State{Solve the Newton equation
			\[
			\calP_{\phibf_k}^{\rm h}\Bigl(\Drm\grad \calE(\phibf_k)[(\psibf_k)_{\phibf_k}^{\rm h}]\Bigr) 
			= -\grad \calE(\phibf_k)
			\]
			\hspace*{5mm} for $(\psibf_k)_{\phibf_k}^{\rm h} \in \calH_{\phibf_k}$ in an inexact manner by using 
			an~iterative solver. If the condition   
			\begin{align*}				
				-\big( \grad \calE(\phibf_k), (\psibf_k)_{\phibf_k}^{\rm h}\big) _H
				&\geq \eta\, \|(\psibf_k)_{\phibf_k}^{\rm h}\|_H^2
			\end{align*} 
			\hspace*{4.4mm} 
		cannot be attained within~$\ell_{\max}$	steps, then set $$(\psibf_k)_{\phibf_k}^{\rm h} = -\grad \calE(\phibf_k).$$
		}
		\State{Find the smallest $\ell\in\N_0$ such that the Armijo condition 
			\begin{align*}
				\calE\big(\calR({\phibf_k},\delta^\ell (\psibf_k)_{\phibf_k}^{\rm h})\big) - \calE\big(\phibf_k\big) 
				\le \sigma \,\delta^\ell
				\big( \grad \calE(\phibf_k), (\psibf_k)_{\phibf_k}^{\rm h}\big)_H
			\end{align*}
			\hspace*{5mm} is satisfied, where $\calR({\phibf_k},\delta^\ell (\psibf_k)_{\phibf_k}^{\rm h})$
			is a retraction on $\StiefelV$.}
		\State{Set $\phibf_{k+1} = \calR\big(\phibf_k,\delta^\ell(\psibf_k)_{\phibf_k}^{\rm h}\big)$.}
		\EndFor
		\State {\bf Output}: sequence of iterates $\{\phibf_k\}$ with $\phibf_k\in\StiefelV$
	\end{algorithmic}
\end{algorithm} 

The Newton equation~\eqref{eq:NewGrV} can also be formulated as a~saddle point problem. To this end, we introduce the~bilinear form 
\[
	\widehat{a}_\phibf(\psibf,\wbf) 
	= \langle \calA_\phibf\,\psibf,\wbf\rangle + \langle \calB_\phibf\, \psibf,\wbf\rangle - (\psibf\, \out{\phibf}{\calJ\!\calA_\phibf\,\phibf},\wbf)_H.  
\]
Then, the equivalent problem to~\eqref{eq:NewGrV} reads: find $(\psibf_k)_{\phibf_k}^{\rm h}\in V$ and a~Lagrange multiplier \mbox{$M_k\in \R^{p\times p}$} such that 
\begin{subequations}
	\label{eq:Saddle2}
	\begin{align}
		\widehat{a}_{\phibf_k}\big((\psibf_k)_{\phibf_k}^{\rm h},\wbf\big)  + \trace \big( M_k^T\out{\phibf_k}{\wbf}\big) & = -a_{\phibf_k}(\phibf_k,\wbf) & & \text{for all } \wbf\in V, \label{eq:Saddle2_1}\\
		\out{\phibf_k}{(\psibf_k)_{\phibf_k}^{\rm h}} \hspace*{32mm}& = 0_p. && \label{eq:Saddle2_2}
	\end{align}
\end{subequations}
The constraint \eqref{eq:Saddle2_2} implies that  
$(\psibf_k)_{\phibf_k}^{\rm h}\in\calH_{\phibf_k}$. Further note that any function $\wbf\in \calH_{\phibf_k}$ satisfies~$\out{\phibf_k}{\wbf}= 0_p$. Hence, 
for all $\wbf\in \calH_{\phibf_k}$, equation~\eqref{eq:Saddle2_1} reads  
$$
	\widehat{a}_\phibf\big((\psibf_k)_{\phibf_k}^{\rm h},\wbf\big) 
	= -a_{\phibf_k}(\phibf_k,\wbf)=-(\grad\calE(\phibf_k),\wbf)_H,
$$ 
which is equivalent to the Newton equation \eqref{eq:NewGrV}.  

One important property guaranteeing an~isolated local minimum of the energy is that the Hessian is positive at a~stationary point. For a global minimizer of~\eqref{eq:minSt}, denoted by~$\phibf_*$, we consider the following {\em linear} eigenvalue problem: seek $\phi \in \tilde{V}$ and $\lambda\in \R$ such that 
\begin{align}
	\label{eq:NLEVPweakComponents:fixedPhi}	
	\tilde a_{\phibf_*}\!(\phi, v) 
	= \lambda\, (\phi, v)_{L^2(\Omega)}
	\qquad\text{ for all } v\in \tilde V.
\end{align}
Then, due to~\eqref{eq:NLEVPweakComponents}, we know that the components of~$\phibf_* = (\phi_{*,1}, \dots, \phi_{*,p})$ satisfy~\eqref{eq:NLEVPweakComponents:fixedPhi} together with the smallest~$p$ eigenvalues denoted by $0 < \lambda_1 \le \dots \le \lambda_p$. In the following, we will assume that these eigenfunctions can be extended to a basis of~$\tilde{V}$. 
\begin{assumption}[Basis and spectral gap]
	\label{ass:gap}
	The eigenfunctions~$\phi_{*,1}, \phi_{*,2}, \ldots \in \tilde{V}$ of the eigenvalue problem~\eqref{eq:NLEVPweakComponents:fixedPhi} form an $L^2$-orthonormal basis of $\tilde{V}$. The corresponding eigenvalues~$\lambda_1 \le \lambda_2 \le \dots$ are ordered by size with a spectral gap~$\lambda_p < \lambda_{p+1}$.
\end{assumption}
\begin{theorem}[Positive Hessian]
	\label{th:positiveHessian}	
	Let~$\phibf_*$ be a global minimal solution of the NLEVP~\eqref{eq:NLEVPop} and let the corresponding eigenvalue problem~\eqref{eq:NLEVPweakComponents:fixedPhi} satisfy Assumption~\textup{\ref{ass:gap}}. Further assume that the nonlinearity fulfills~$(\calJ \calB_{\phibf_*} \psibf_{\phibf_*}^{\rm h}, \psibf_{\phibf_*}^{\rm h})_H \ge 0$ for all $\psibf_{\phibf_*}^{\rm h}\in\calH_{\phibf_*}$. Then the Riemannian Hessian of~$\calF$ at $\equivcl{\phibf_*}$ is positive, i.e., 
	\[
	g^{\rm Gr}\big( \Hess{\calF}(\equivcl{\phibf_*})[\psibf], \psibf \big)
	> 0
	\]
	for all nonzero~$\psibf\in T_{\equivcl{\phibf_*}}\GrassV$. 
\end{theorem}
\begin{proof}
	We extend the proof of~\cite[Th.~5.1]{ZhaoBJ15}, which considers the finite-dimensional case for the simplified Kohn--Sham problem, to the infinite-dimensional case in a more general setting. We know from~\eqref{eq:HessGrassmann} that the horizontal lift of the Riemannian Hessian of $\calF$ at $\phibf_*$ takes the form 
	\begin{align*}
		(\Hess{\calF}(\equivcl{\phibf_*})[\psibf])_{\phibf_*}^{\rm h} 
		&= \calP_{\phibf_*}^{\rm h} \big(\calJ\!\calA_{\phibf_*} \psibf_{\phibf_*}^{\rm h}\big) 
		+ \calP_{\phibf_*}^{\rm h}\big( \calJ\calB_{\phibf_*} \psibf_{\phibf_*}^{\rm h}\big) 
		-  \calP_{\phibf_*}^{\rm h}\big(\psibf_{\phibf_*}^{\rm h}\out{\phibf_*}{\calJ\!\calA_{\phibf_*} \phibf_*} \big) \\
		&= T_1(\psibf_{\phibf_*}^{\rm h}) 
		\hspace{1.27cm} + T_2(\psibf_{\phibf_*}^{\rm h}) 
		\hspace{1.31cm}+ T_3(\psibf_{\phibf_*}^{\rm h}). 
	\end{align*}
	For the first term, we make the following considerations. Due to the $H$-orthogonality, each component of $\psibf_{\phibf_*}^{\rm h}$ satisfies $(\phi_{*,l}, (\psibf_{\phibf_*}^{\rm h})_j)_{L^2(\Omega)} = 0$ for $l,j= 1,\dots, p$. Hence, the $j$th component of $\psibf_{\phibf_*}^{\rm h}$ takes the form $\psi_j = (\psibf_{\phibf_*}^{\rm h})_j = \sum_{l>p} \alpha_{lj} \phi_{*,l}$ for some coefficients~$\alpha_{lj}\in\R$ and~$\phi_{*,l}$ denoting the basis from Assumption~\ref{ass:gap}. As a consequence, we get 
	\begin{align*}
		\big( \out{\phibf_*}{\calJ\!\calA_{\phibf_*} \psibf_{\phibf_*}^{\rm h}} \big)_{ij}
		& = \big( \phi_{*,i}, (\calJ\!\calA_{\phibf_*} \psibf_{\phibf_*}^{\rm h})_j \big)_{L^2(\Omega)}
		= \tilde{a}_{\phibf_*}\!(\phi_{*,i}, \psi_j) \\
		& = \sum_{l>p} \alpha_{lj}\, \tilde{a}_{\phibf_*}\!(\phi_{*,i}, \phi_{*,l})
		= \sum_{l>p} \alpha_{lj}\, \lambda_i (\phi_{*,i}, \phi_{*,l})_{L^2(\Omega)}
		= 0
	\end{align*}
	for all $i,j=1,\ldots, p$ and, hence, 
	\[
	T_1(\psibf_{\phibf_*}^{\rm h})
	= \calP_{\phibf_*}^{\rm h} \big(\calJ\!\calA_{\phibf_*} \psibf_{\phibf_*}^{\rm h}  \big)  
	= \calJ\!\calA_{\phibf_*} \psibf_{\phibf_*}^{\rm h} - \phibf_* 
	\out{\phibf_*}{\calJ\!\calA_{\phibf_*} \psibf_{\phibf_*}^{\rm h}	}
	= \calJ\!\calA_{\phibf_*} \psibf_{\phibf_*}^{\rm h}. 
	\]
	For the second term, we get with the assumption on~$\calB_{\phibf_*}$ that 
	\begin{align*}
		g( T_2(\psibf_{\phibf_*}^{\rm h}), \psibf_{\phibf_*}^{\rm h} )
		&= \big(T_2(\psibf_{\phibf_*}^{\rm h}), \psibf_{\phibf_*}^{\rm h}\big)_H \\
		&= \big(\calJ \calB_{\phibf_*} \psibf_{\phibf_*}^{\rm h} - \phibf_* \out{\phibf_*}{\calJ \calB_{\phibf_*} \psibf_{\phibf_*}^{\rm h}},\psibf_{\phibf_*}^{\rm h}\big)_H  \\
		&= \big(\calJ \calB_{\phibf_*} \psibf_{\phibf_*}^{\rm h},\psibf_{\phibf_*}^{\rm h}\big)_H  - \out{\phibf_*}{\calJ \calB_{\phibf_*} \psibf_{\phibf_*}^{\rm h}}^T \big(\phibf_*,\psibf_{\phibf_*}^{\rm h}\big)_H\\
		&= \big(\calJ \calB_{\phibf_*} \psibf_{\phibf_*}^{\rm h},\psibf_{\phibf_*}^{\rm h}\big)_H 
		\ge 0.   
	\end{align*}
	Finally, by using \eqref{eq:Lambda}, the third term takes the form  
	\begin{align*}
		T_3(\psibf_{\phibf_*}^{\rm h})
		&= - \calP_{\phibf_*}^{\rm h} \big( \psibf_{\phibf_*}^{\rm h} \out{\phibf_*}{\calJ\!\calA_{\phibf_*}\phibf_*} \big) \\
		&= - \psibf_{\phibf_*}^{\rm h} \out{\phibf_*}{\calJ\!\calA_{\phibf_*}\phibf_*} 
		+ \phibf_* \out{\phibf_*}{\psibf_{\phibf_*}^{\rm h}}\, \out{\phibf_*}{\calJ\!\calA_{\phibf_*}\phibf_*} 
		= - \psibf_{\phibf_*}^{\rm h} \Lambda_*.
	\end{align*}
	Let the columns of $U\in\OrthGr$ form a~basis of eigenvectors corresponding to the eigenvalues $\lambda_1,\ldots,\lambda_p$ of $\Lambda_*$ and let $\psibf_{\phibf_*}^{\rm h}U=(\tilde{\psi}_1,\ldots,\tilde{\psi_p})$. Due to the assumed spectral gap, this yields all together  
	\begin{align*}
		g^{\rm Gr}\big( \Hess{\calF}(\equivcl{\phibf_*})[\psibf], \psibf \big) & = g\big( (\Hess{\calF}(\equivcl{\phibf_*})[\psibf])_{\phibf_*}^{\rm h}, \psibf_{\phibf_*}^{\rm h} \big) \\
		&= g\big( T_1(\psibf_{\phibf_*}^{\rm h}) + T_2(\psibf_{\phibf_*}^{\rm h}) + T_3(\psibf_{\phibf_*}^{\rm h}), \psibf_{\phibf_*}^{\rm h} \big) \\
		&\ge \trace\, \out{\calJ\!\calA_{\phibf_*} \psibf_{\phibf_*}^{\rm h}}{\psibf_{\phibf_*}^{\rm h}} 
		- \trace\, \out{\psibf_{\phibf_*}^{\rm h} \Lambda_*}{\psibf_{\phibf_*}^{\rm h}} 		\\
		&= \sum_{j=1}^p \tilde{a}_{\phibf_*}\!(\tilde{\psi}_j, \tilde{\psi}_j) 
		- \sum_{j=1}^p \lambda_j (\tilde{\psi}_j, \tilde{\psi}_j)_{L^2(\Omega)} \\
		&\ge \sum_{j=1}^p ( \lambda_{p+1} - \lambda_j)\, (\tilde{\psi}_j, \tilde{\psi}_j)_{L^2(\Omega)}
		> 0,
	\end{align*}
	which completes the proof.  
\end{proof}
\begin{remark}[Connection to the Lagrange--Newton method]
The optimal solution of the constrained minimization problem~\eqref{eq:minSt} can also be determined by the Lagrange--Newton method. Based on the first-order optimality conditions~\eqref{eq:NLEVPop} with a symmetric Lagrange multiplier, we aim to solve the nonlinear system of equations
\[
f(\phibf, \Lambda) = \begin{bmatrix} \calJ\!\calA_{\phibf} \,\phibf-\phibf\, \Lambda \\ 
			\out{\phibf}{\phibf} - I_p \\ \Lambda-\Lambda^T\end{bmatrix} 
			= \begin{bmatrix} \mathbf{0}\, \\ 0_p \\ 0_p\end{bmatrix}.
\]
Computing the Jacobian of $f$, the Lagrange-Newton iteration is given as follows: for given $\phibf_k\in V$ and $\Lambda_k\in\R^{p\times p}$, solve the equations
\begin{subequations}
	\label{eq:LagrNewton}
	\begin{align}
	\calJ\!\calA_{\phibf_k} \etabf_k + \calJ\calB_{\phibf_k} \etabf_k 
	-\etabf_k\Lambda_k-\phibf_k\Xi_k 
	&= -\big(\calJ\!\calA_{\phibf_k} \phibf_k -\phibf_k\Lambda_k\big), \label{eq:LagrNewton1} \\
	\out{\phibf_k}{\etabf_k}+\out{\etabf_k}{\phibf_k} 
	&= -\big(\out{\phibf_k}{\phibf_k} - I_p\big), \label{eq:LagrNewton2}\\
	\Xi_k^{}-\Xi_k^T 
	&= -\big(\Lambda_k^{}-\Lambda_k^T\big) \label{eq:LagrNewton3}
	\end{align}
\end{subequations}
for $\etabf_k\in V$, $\Xi_k\in\R^{p\times p}$ and update $\phibf_{k+1}=\phibf_k+\etabf_k$, $\Lambda_{k+1}=\Lambda_k+\Xi_k$. Note that $\phibf_{k+1}$ does not necessarily belong to $\StiefelV$. Assuming $\phibf_k\in\StiefelV$ and $\Lambda_k\in\Ssym$, however, equations \eqref{eq:LagrNewton2} and \eqref{eq:LagrNewton3} imply that $\etabf_k\in T_{\phibf_k}\StiefelV$ and $\Xi_k\in\Ssym$, respectively. Resolving equation \eqref{eq:LagrNewton1} for symmetric $\Xi_k$, we find that
\[
	\Xi_k = \sym\big(\out{\phibf_k}{\calJ\!\calA_{\phibf_k}\,\phibf_k}+\out{\phibf_k}{\calJ\calB_{\phibf_k}\,\phibf_k}-
	\out{\phibf_k}{\etabf_k}\Lambda_k\big)+\out{\phibf_k}{\calJ\!\calA_{\phibf_k}\,\phibf_k}-\Lambda_k.
\]
Inserting this matrix into \eqref{eq:LagrNewton1} yields the Newton equation~\eqref{eq:NewtonEqSt}. This shows that the Lagrange--Newton method with the modified update 
\[
	\phibf_{k+1} 
	= \calR(\phibf_k,\etabf_k), \qquad 
	\Lambda_{k+1} 
	= \out{\phibf_{k+1}}{\calJ\!\calA_{\phibf_{k+1}}\,\phibf_{k+1}}
\] 
is equivalent to the Newton method on the Stiefel manifold.
\end{remark}
%
%
\section{Examples and numerical experiments}\label{sect:numerics}
This section is devoted to the numerical investigation of the Riemannian Newton methods. To this end, we consider the Gross--Pitaevskii eigenvalue problem from Example~\ref{exp:GPEVP} and the Kohn--Sham model from Example~\ref{exp:KS}. 
%
%
\subsection{Gross--Pitaevskii eigenvalue problem}\label{sect:numerics:GPEVP}
The minimization of the Gross--Pitaevskii energy functional $\calE_{\rm GP}$ in~\eqref{eq:energyGP} leads to the following nonlinear eigenvector problem: find $\phi\in \tilde{V} = H^1_0(\Omega)$ with $\|\phi\|_{L^2(\Omega)} = 1$ and $\lambda\in\R$ such that 
\begin{equation}\label{eq:GPeig}
-\Delta \phi + 2\vartheta\, \phi + \kappa\, |\phi|^2 \phi 
= \lambda\, \phi 
\end{equation}
for some space-dependent external potential~$\vartheta \ge 0$ and an interaction constant $\kappa>0$. The latter means that the particle interactions are repulsive, i.e., we consider the so-called {\em defocussing regime}. In this case, we get the operators
\begin{subequations}
	\begin{align}
		\langle \calA_\phi\, v, w\rangle
		&= \int_\Omega (\nabla v)^T \nabla w \dx 
		+ 2 \int_\Omega  \vartheta\, v w \dx 
		+ \kappa\int_\Omega \, \phi^2 v w \dx, \label{eq:Aphi-GP} \\
		\langle \calB_\phi\, v, w\rangle 
		&= 2\kappa \int_{\Omega} \phi^2 v w \dx \label{eq:Bphi-GP} 
	\end{align}
\end{subequations}
for $v,w\in \tilde{V}$. One can see that the bilinear form defined through \eqref{eq:Aphi-GP} corresponds to the Laplacian with the $L^2$-shift~$2\vartheta+\kappa\phi^2$. Assuming this shift to be constant and $\Omega=(0,1)^d$ as the spatial domain, Assumption~\ref{ass:gap} is satisfied; see~\cite[Ch.~12]{Wlo87}. 
Moreover, the nonlinear operator from~\eqref{eq:Bphi-GP} fulfills  
\[
\big( \calJ \calB_{\phi}\, \psi, \psi \big)_{L^2(\Omega)}
= \big\langle \calB_{\phi}\,\psi, \psi \big\rangle 
= 2\kappa \int_{\Omega} \phi^2\, \psi^2 \dx 
\ge 0,
\]	
such that Theorem~\ref{th:positiveHessian} is applicable.

For the spatial discretization of the Gross--Pitaevskii problem \eqref{eq:GPeig}, we use a~biquadratic finite element method on a~Cartesian mesh of width $h$; see~\cite{CanCM10} for the corresponding error analysis. The resulting discrete eigenvalue problem reads 
\[
A \varphi + 2M_{\vartheta} \varphi + \kappa M_{\varphi^2} \varphi 
= \lambda\, M \varphi, \qquad \varphi^TM\varphi =1 
\]
with $\varphi\in \R^n$, where $n$ denotes the number of degrees of freedom. Here, $A$  is the stiffness matrix, $M$ is the mass matrix, and $M_{\vartheta}$ and $M_{\varphi^2}$ are the weighted mass matrices, respectively, where~$\varphi^2$ should be understood as the elementwise product.
Then the discrete version of~$\out{\phi}{\calJ\!\calA_\phi\,\phi}=\big(\phi,\calJ\!\calA_\phi\,\phi\big)_{L^2(\Omega)}$ equals 
\[
\lambda_\varphi = \varphi^T (A + 2M_{\vartheta} + \kappa M_{\varphi^2})\, \varphi,
\]
and the Newton equation takes the form
\begin{align*}
	(I - M\varphi \varphi^T)\big((A + 2M_{\vartheta} + 3\kappa M_{\varphi^2})\, \psi - \lambda_\varphi M \psi\big) 
	= - (I - M\varphi \varphi^T)(A + 2M_{\vartheta} + \kappa M_{\varphi^2})\, \varphi
\end{align*}
with unknown $\psi\in \{\xi\in \R^n\;:\; \xi^TM\varphi = 0\}=\im (I-\varphi\varphi^T\! M)$.

We demonstrate the performance of the resulting Riemannian Newton method in comparison with the SCF iteration combined with the optimal damping algorithm (ODA) proposed in~\cite{DioC07} and the energy-adapted Riemannian gradient descent method (RGD) of~\cite{HenP20} with a non-monotone step size control as outlined in~\cite{AltPS21}. The numerical experiments are performed on a~sufficiently large bounded domain  $\Omega = (-L,L)^2$, $L=8$, for two types of trapping potentials. In Section~\ref{sect:numerics:GPEVP:exp1}, we consider a simple harmonic trap, whereas in Section~\ref{sect:numerics:GPEVP:exp2}, we add an~additional disorder potential. The interaction parameter~$\kappa$ as well as the spatial resolution~$h$ will be specified below, separately for each case. 
%
%
\subsubsection{Ground state in a harmonic trap}\label{sect:numerics:GPEVP:exp1} 
For the first numerical experiment, we consider the harmonic trapping potential
\begin{equation}\label{eq:Vharm}
\vartheta_\text{harm}(x)  
= \tfrac{1}{2}\|x\|^2
\end{equation}
and interaction parameters $\kappa=10,100,1000$. The resulting ground states computed on a Cartesian mesh of width $h/(2L) = 2^{-10}$ are depicted in Figure~\ref{fig:groundstateHarmonic}. 
\begin{figure}
	\centering
	\includegraphics[width=0.3\linewidth]{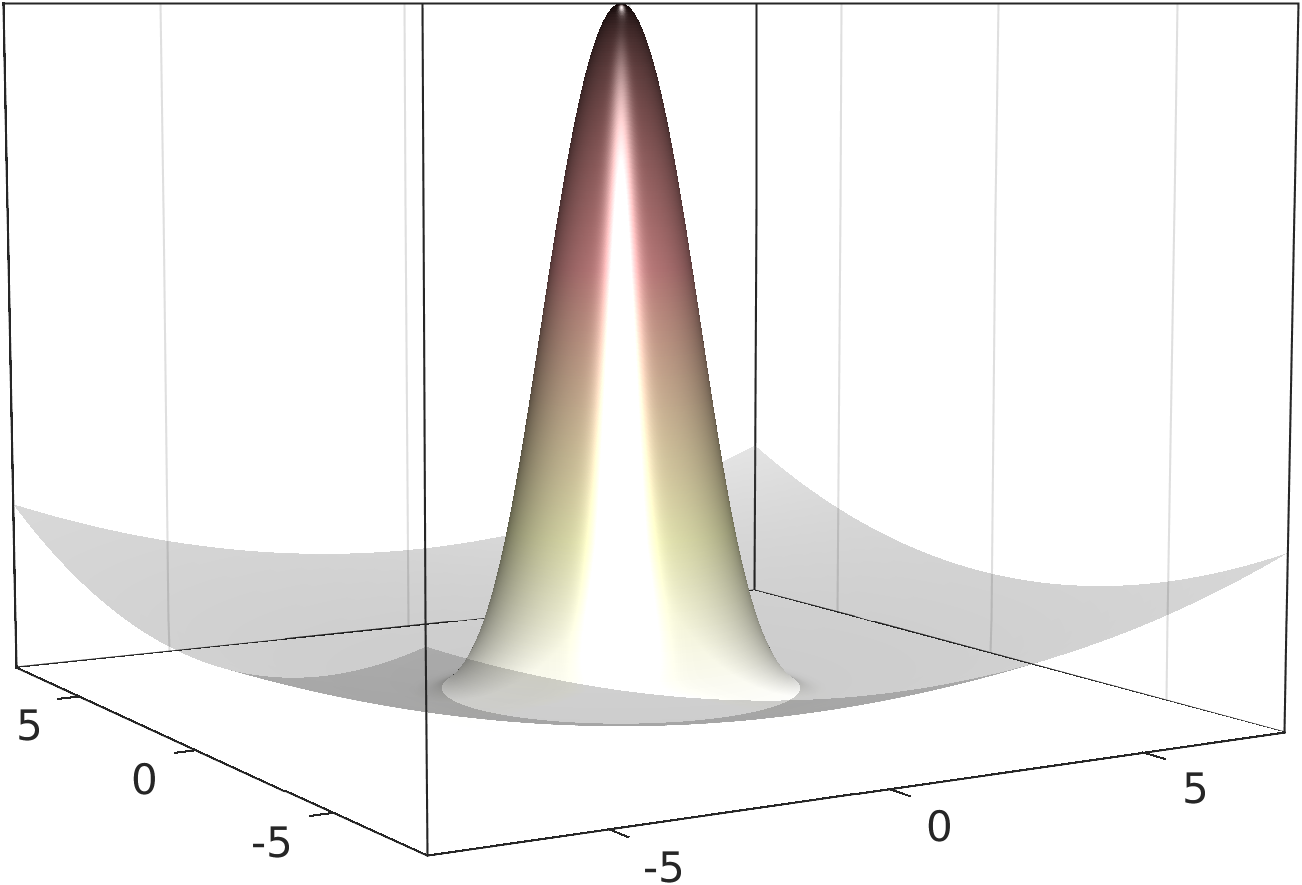}\quad
	\includegraphics[width=0.3\linewidth]{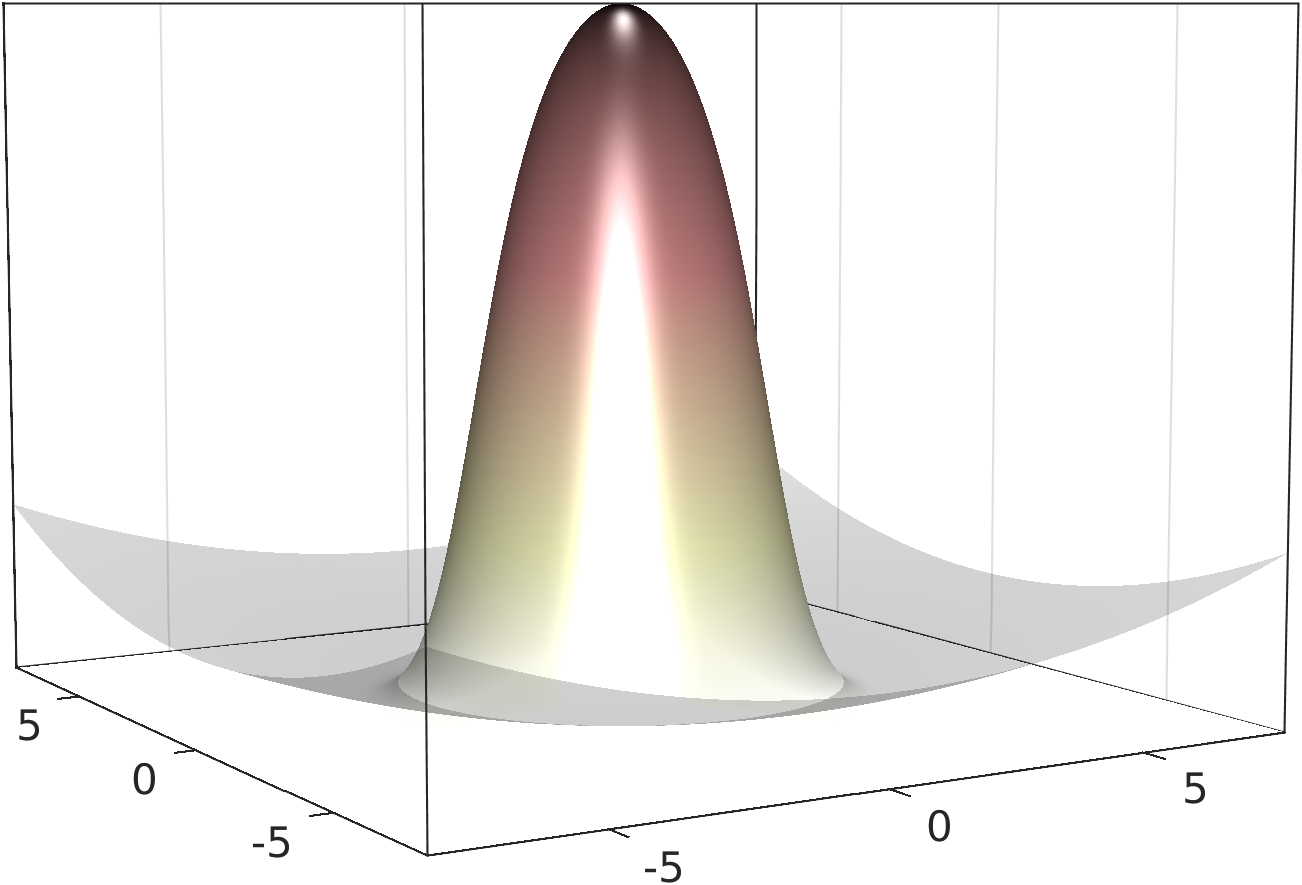}\quad
	\includegraphics[width=0.3\linewidth]{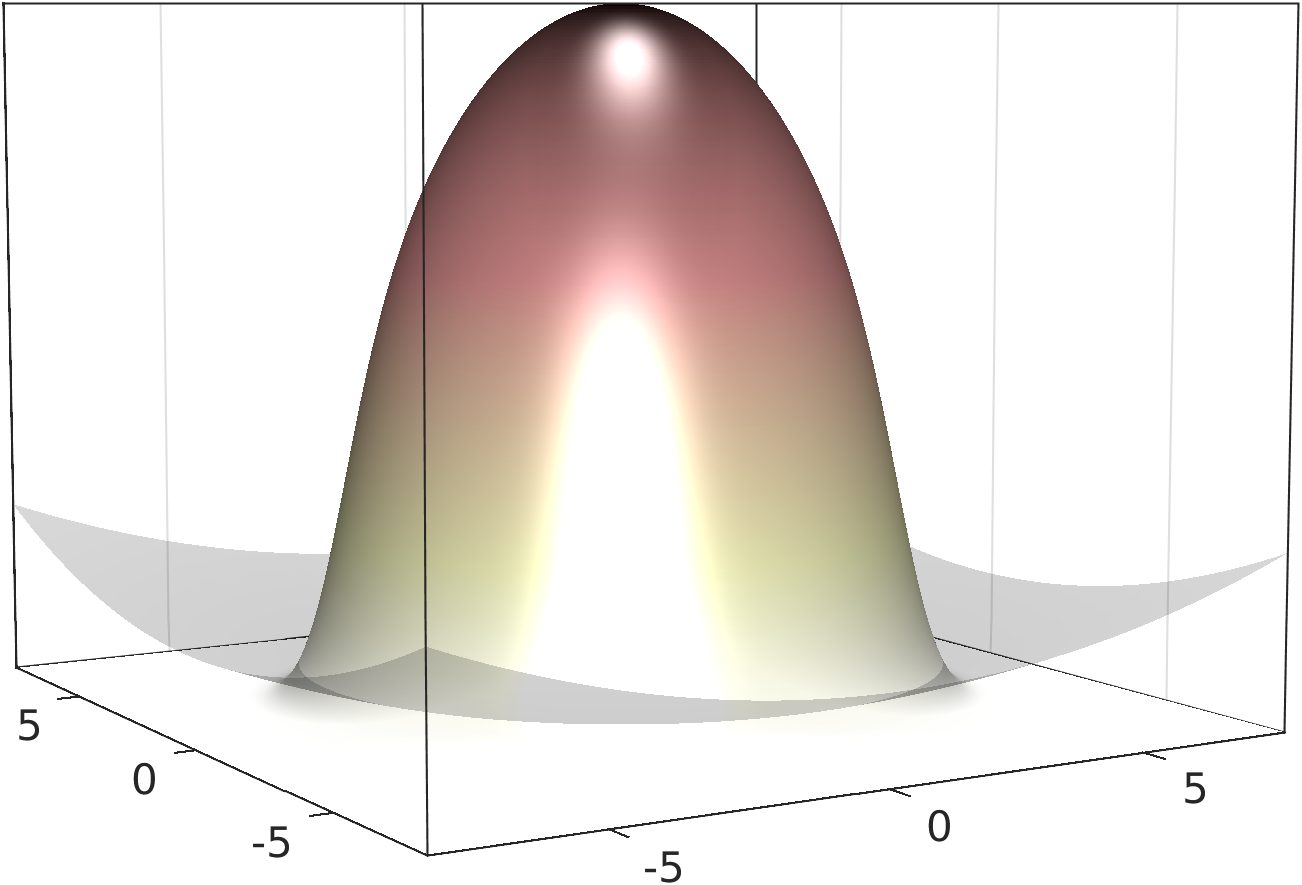}
	\caption{Ground state in the harmonic trap (potential in gray, properly rescaled) for $\kappa=10,100,1000$ (from left to right).}
	\label{fig:groundstateHarmonic}
\end{figure}
To generate a~joint and sufficiently accurate initial value for the three solvers of the discretized nonlinear eigenvector problem, we run the energy-adapted RGD method starting from the biquadratic finite element interpolation of the constant $1$ (respecting the homogeneous Dirichlet boundary condition). We stopped this iteration once the residual fell below the $10^{-2}$ tolerance and used the approximated ground state as the initial state to compare the asymptotic behavior of the three different solvers. The corresponding convergence histories are presented in Figure~\ref{fig:convergenceHarmonic} showing the evolution of the residuals during the iteration processes. It can be observed that the Riemannian Newton method (with sparse direct solution of the Newton equation using the Sherman--Morrison formula~\cite{SheM50}) reaches the tolerance of $10^{-8}$ in only three steps. While the performances of the SCF iteration and the energy-adaptive RGD method abate with increasing $\kappa$, the Riemannian Newton scheme appears to be extremely robust. 
We would like to emphasize that, although one Newton step is slightly more expensive than one step of any other competing method, the overall costs are much smaller, especially for increasing~$\kappa$. 

\begin{figure}
%
%
\begin{tikzpicture}

\begin{axis}[%
width=1.5in,
height=1.4in,
at={(0.0in,0.0in)},
scale only axis,
xmin=0.0,
xmax=10,
xlabel={iteration steps},
ymode=log,
ymin=1.1e-10,
ymax=0.01,
yticklabel={\empty},
yminorticks=true,
ylabel={residual},
ylabel near ticks,
axis background/.style={fill=white},
legend columns = 3,
legend style={legend cell align=left, align=left, at={(1.83,1.05)}, anchor=south, draw=white!15!black}
]
\addplot [color=mycolor4, line width=1.25pt, mark=square*]
  table[row sep=crcr]{%
0	0.00408183409999703\\
1	8.8443858587292e-06\\
2	1.80108216700751e-08\\
3	2.89302249402429e-13\\
};
\addlegendentry{Riemannian Newton\qquad}

\addplot [color=mycolor2, line width=1.25pt, mark=triangle*]
  table[row sep=crcr]{%
0	0.00408183409999703\\
1	5.77458363720633e-06\\
2	3.19293440539398e-06\\
3	1.6851548784485e-06\\
4	7.89970674131212e-07\\
5	3.96225554973076e-07\\
6	1.7788936112595e-07\\
7	8.89674347415363e-08\\
8	3.99223842415537e-08\\
9	1.99619826265496e-08\\
10	8.95807760731701e-09\\
};
\addlegendentry{SCF (ODA)\qquad}

\addplot [color=mycolor3, line width=1.25pt, mark=*]
  table[row sep=crcr]{%
0	0.00408183409999703\\
1	0.00404104865777588\\
2	0.000599276455279867\\
3	0.000139727345167003\\
4	6.07150937236149e-05\\
5	3.12830999863435e-05\\
6	5.43045146156012e-06\\
7	2.4324557419963e-07\\
8	9.67749649831225e-09\\
};
\addlegendentry{energy-adapted RGD} 

\end{axis}

%
%

\begin{axis}[%
width=1.5in,
height=1.4in,
at={(2.0in,0.0in)},
scale only axis,
xmin=0.0,
xmax=20,
xlabel={iteration steps},
ymode=log,
ymin=1.1e-10,
ymax=0.01,
yminorticks=true,
axis background/.style={fill=white},
]
\addplot [color=mycolor4, line width=1.25pt, mark=square*]
  table[row sep=crcr]{%
0	0.00490544064297205\\
1	2.15259919706661e-05\\
2	2.24174953294405e-07\\
3	2.14235478864784e-11\\
};

\addplot [color=mycolor2, line width=1.25pt, mark=triangle*]
  table[row sep=crcr]{%
0	0.00490544064297205\\
1	8.53970142410399e-05\\
2	5.39037399588439e-05\\
3	3.05083148488343e-05\\
4	2.72988706273288e-05\\
5	1.10386006782954e-05\\
6	1.10175699769235e-05\\
7	6.03158729481822e-06\\
8	7.60982501316957e-06\\
9	3.79696012334323e-06\\
10	2.42164132981341e-06\\
11	5.23174214089027e-07\\
12	2.72261710351939e-07\\
13	1.74913937843423e-07\\
14	1.04754057451453e-07\\
15	8.01908171827973e-08\\
16	3.81448977828642e-08\\
17	3.00795221280715e-08\\
18	3.02526193093071e-08\\
19	2.10326533883803e-08\\
20	2.92555157444167e-08\\
21	1.39946910002837e-08\\
22	9.41388192976028e-09\\
};

\addplot [color=mycolor3, line width=1.25pt, mark=*]
  table[row sep=crcr]{%
0	0.00490544064297205\\
1	0.00485648970508212\\
2	0.00116413599801866\\
3	0.000416598091704904\\
4	0.000133402357443834\\
5	7.279894426898e-06\\
6	2.00852284397611e-06\\
7	8.14093654016487e-07\\
8	3.03463505059465e-07\\
9	2.16100211530998e-07\\
10	1.33352506614173e-07\\
11	2.71321318568083e-08\\
12	8.86283557386394e-09\\
};

\end{axis}
%
%

\begin{axis}[%
width=1.5in,
height=1.4in,
at={(4.0in,0.0in)},
scale only axis,
xmin=0.0,
xmax=30,
xlabel={iteration steps},
ymode=log,
ymin=1.1e-10,
ymax=0.01,
yminorticks=true,
axis background/.style={fill=white},
]
\addplot [color=mycolor4, line width=1.25pt, mark=square*]
  table[row sep=crcr]{%
0	0.0021501935728488\\
1	5.04414619892899e-05\\
2	1.74274903320897e-06\\
3	8.61937473304991e-09\\
};

\addplot [color=mycolor2, line width=1.25pt, mark=triangle*]
  table[row sep=crcr]{%
0	0.0021501935728488\\
1	0.000123158573121054\\
2	0.000177873954957721\\
3	0.000154033394913062\\
4	8.098431578547e-05\\
5	0.000132413302957006\\
6	0.000134649442293355\\
7	6.99736771854764e-05\\
8	9.71100728467796e-05\\
9	0.000157064883899476\\
10	0.000137726065836452\\
11	8.68583159735097e-05\\
12	8.19794783764031e-05\\
13	4.91030161615511e-05\\
14	6.3313491548064e-05\\
15	0.000115251746942936\\
16	0.000133881133364934\\
17	4.57548303917571e-05\\
18	5.01512110300377e-05\\
19	4.52112316745892e-05\\
20	0.000131229534591911\\
21	0.000150525005298611\\
22	4.48354930120353e-05\\
23	4.52824810536671e-05\\
24	3.17159232421044e-05\\
25	3.83524893014111e-05\\
26	3.09544797218427e-05\\
27	3.11794386093035e-05\\
28	6.04379677741512e-05\\
29	3.78241459590894e-05\\
30	3.15906819964814e-05\\
31	1.54094499951244e-05\\
32	2.48688707282603e-05\\
33	1.7795837070284e-05\\
34	3.53070978491408e-05\\
35	1.96030750110579e-05\\
36	1.68822382091727e-05\\
37	1.98305420142684e-05\\
38	1.49097904862176e-05\\
39	2.96698933799893e-05\\
40	2.0247001303632e-05\\
41	1.95261274683078e-05\\
42	1.50219012545263e-05\\
43	3.49324314779954e-05\\
44	1.6991359367826e-05\\
45	1.38604943620061e-05\\
46	1.38334961880011e-05\\
47	1.23708226527248e-05\\
48	1.86887961859534e-05\\
49	1.38179422593034e-05\\
50	2.80666864701684e-05\\
};

\addplot [color=mycolor3, line width=1.25pt, mark=*, mark repeat=2]
  table[row sep=crcr]{%
0	0.0021501935728488\\
1	0.0021291682171424\\
2	0.00116811589819549\\
3	0.00092862497103897\\
4	9.62418036640646e-05\\
5	3.47524332258158e-05\\
6	1.93748988018166e-05\\
7	1.17155258053568e-05\\
8	8.27370599077937e-06\\
9	7.1727132150353e-06\\
10	2.65332671169553e-06\\
11	1.44059636274283e-06\\
12	1.15095815872819e-06\\
13	6.27435681775532e-07\\
14	3.69951609164271e-07\\
15	3.60759972729042e-07\\
16	2.1386330882466e-07\\
17	1.20842644100953e-07\\
18	9.77804936061327e-08\\
19	5.31204473221539e-08\\
20	3.15032893455613e-08\\
21	3.3993980946963e-08\\
22	1.75821707437573e-08\\
23	1.05351035004685e-08\\
24	8.72767848664894e-09\\
};

\end{axis}
\end{tikzpicture}%
	\caption{Convergence history of the residuals for the ground state in the harmonic trap for $\kappa=10,100,1000$ (from left to right). }
	\label{fig:convergenceHarmonic}
\end{figure}
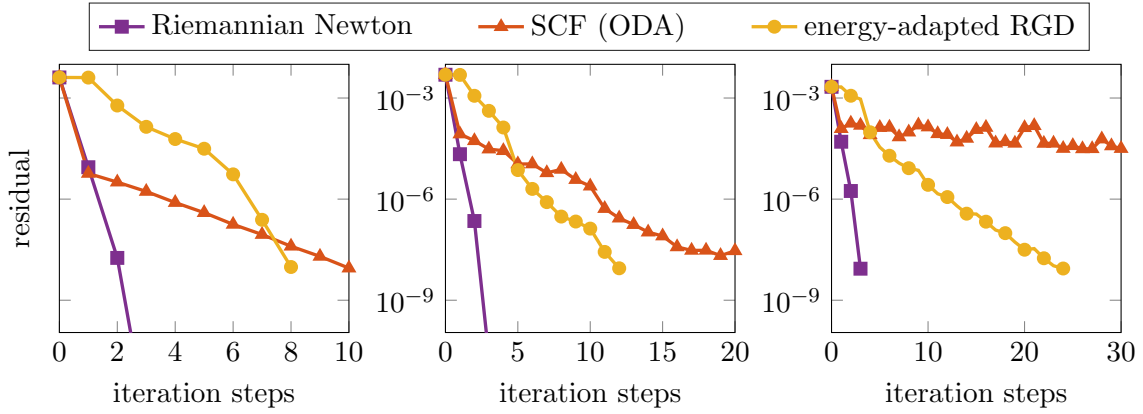
The convergence behavior of the Riemannian Newton scheme is also robust to the underlying mesh size $h$ and, hence, independent of the dimension of the discretization space. This is demonstrated in Figure~\ref{fig:meshind} with a fixed choice of $\kappa=1000$. We consider a sequence of meshes with $h/(2L) = 2^{-1},\ldots,2^{-10}$ and use the same procedure as above to generate initial guesses with residuals of order $10^{-2}$. The left graph shows the number of (outer) iterations of the Riemannian Newton method to fall below the tolerance of $10^{-10}$ for each of these mesh sizes. An increase in the number of iterations with smaller mesh size is not observed. In our experience, this mesh independence of the Riemannian Newton optimization scheme is representative for many other choices of potentials and interaction parameters. Note, however, that this does not mean that the costs of a Newton step are independent of~$h$. As for every Laplace-type problem, methods such as multigrid need to be implemented in order to obtain a mesh-independence also for the inner iteration. This holds for all competing methods in the same way. For completeness, Figure~\ref{fig:meshind} also shows the corresponding errors in the minimum energy approximation as a function of the mesh size, demonstrating the optimal fourth-order convergence rate of the biquadratic finite element implementation~\cite{HenY24}. 
\begin{figure}
%
%
\begin{tikzpicture}

\begin{axis}[%
width=2.2in,
height=1.5in,
at={(0.0in,0.0in)},
scale only axis,
xmode=log,
xmin=0.0008,
xmax=0.58,
xlabel={mesh size~$h$},
ymin=0.4,
ymax=4.6,
ylabel near ticks,
ytick={1, 2, 3, 4},
ylabel={$\#$ Newton steps},
yminorticks=true,
axis background/.style={fill=white},
]
\addplot [color=mycolor4, line width=1.25pt, mark=square*]
  table[row sep=crcr]{%
0.5	1\\
0.25	4\\
0.125	2\\
0.0625	2\\
0.03125	2\\
0.015625	3\\
0.0078125	3\\
0.00390625	3\\
0.001953125	3\\
0.0009765625	3\\
};

\end{axis}

%
%

\begin{axis}[%
width=2.2in,
height=1.5in,
at={(3.0in,0.0in)},
scale only axis,
xmode=log,
xmin=0.0008,
xmax=0.58,
xlabel={mesh size~$h$},
ymode=log,
ymin=1.1e-11,
ymax=1.2,
yminorticks=true,
ylabel={error in min.~energy},
axis background/.style={fill=white},
]
\addplot [color=mycolor4, line width=1.25pt, mark=square*]
  table[row sep=crcr]{%
0.5	0.820610320871369\\
0.25	0.184644385494066\\
0.125	0.0168166218939145\\
0.0625	0.00151041329402979\\
0.03125	0.000123253358859543\\
0.015625	8.40447333594341e-06\\
0.0078125	5.3786262199651e-07\\
0.00390625	3.38234436014773e-08\\
0.001953125	2.11987050136031e-09\\
0.0009765625	1.32862609802942e-10\\
};

\addplot [color=gray, dashed]
  table[row sep=crcr]{%
0.5	0.0625\\
0.25	0.00390625\\
0.125	0.000244140625\\
0.0625	1.52587890625e-05\\
0.03125	9.5367431640625e-07\\
0.015625	5.96046447753906e-08\\
0.0078125	3.72529029846191e-09\\
0.00390625	2.3283064365387e-10\\
0.001953125	1.45519152283669e-11\\
0.0009765625	9.09494701772928e-13\\
};

\end{axis}
\end{tikzpicture}%
	\caption{Computing the ground state in the harmonic trap for \mbox{$\kappa=1000$}: Iteration count of the Riemannian Newton method to fall below the tolerance $10^{-10}$ (left) and the error in the minimal energy (right) versus mesh size $h$. The dashed line indicates order $h^4$. }
	\label{fig:meshind}
\end{figure}
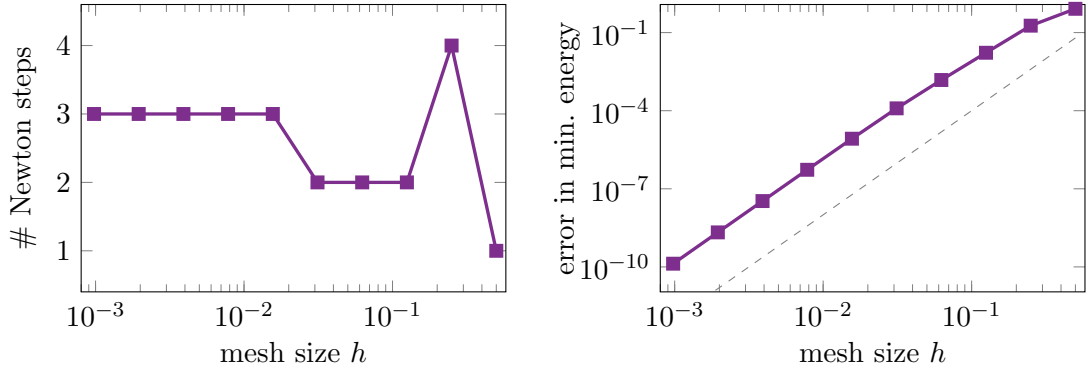

\subsubsection{Localized ground state in a disorder potential}\label{sect:numerics:GPEVP:exp2}
The second experiment considers the computationally more difficult case where the external potential is the sum of the harmonic potential $\vartheta_\text{harm}$ defined in~\eqref{eq:Vharm} and a~potential $\vartheta_\text{rand}$ reflecting a high degree of disorder. The disorder part $\vartheta_\text{rand}$ is chosen as a~piecewise constant function on the Cartesian mesh of width $2L\varepsilon$, $\varepsilon = 2^{-6}$, taking values $0$ or $\varepsilon^{-2}$ as depicted in Figure~\ref{fig:groundstateDisorder}.
For a potential in such a scaling regime, the low-energy eigenstates essentially localize in terms of an~exponential decay of their moduli relative to the small parameter $\varepsilon$. For the linear case, i.e., for~$\kappa=0$, this has been analyzed in~\cite{AltHP20}. For growing~$\kappa$, the ground state consists of a growing number of localized peaks; see~Figure~\ref{fig:groundstateDisorder}. Further details on the phenomenon of localization in the Gross--Pitaevskii equation and the onset of delocalization can be found in~\cite{AltP19,AltHP22}.
As in the previous experiments, we use biquadratic finite elements on a Cartesian mesh of width $h/(2L) = 2^{-10}$. To illustrate the localization behavior that occurs with the current parameter scaling for $\kappa\lesssim 1$, we consider the interaction parameters $\kappa=0.1,1,10$. The ground states for $\kappa=1,10$ are shown in Figure~\ref{fig:groundstateDisorder}. The ground state for $\kappa=0.1$ is hardly distinguishable from the one for $\kappa=1$ and, therefore, it is not shown in a~separate figure. 
\begin{figure}
	\centering
	\includegraphics[width=0.23\linewidth]{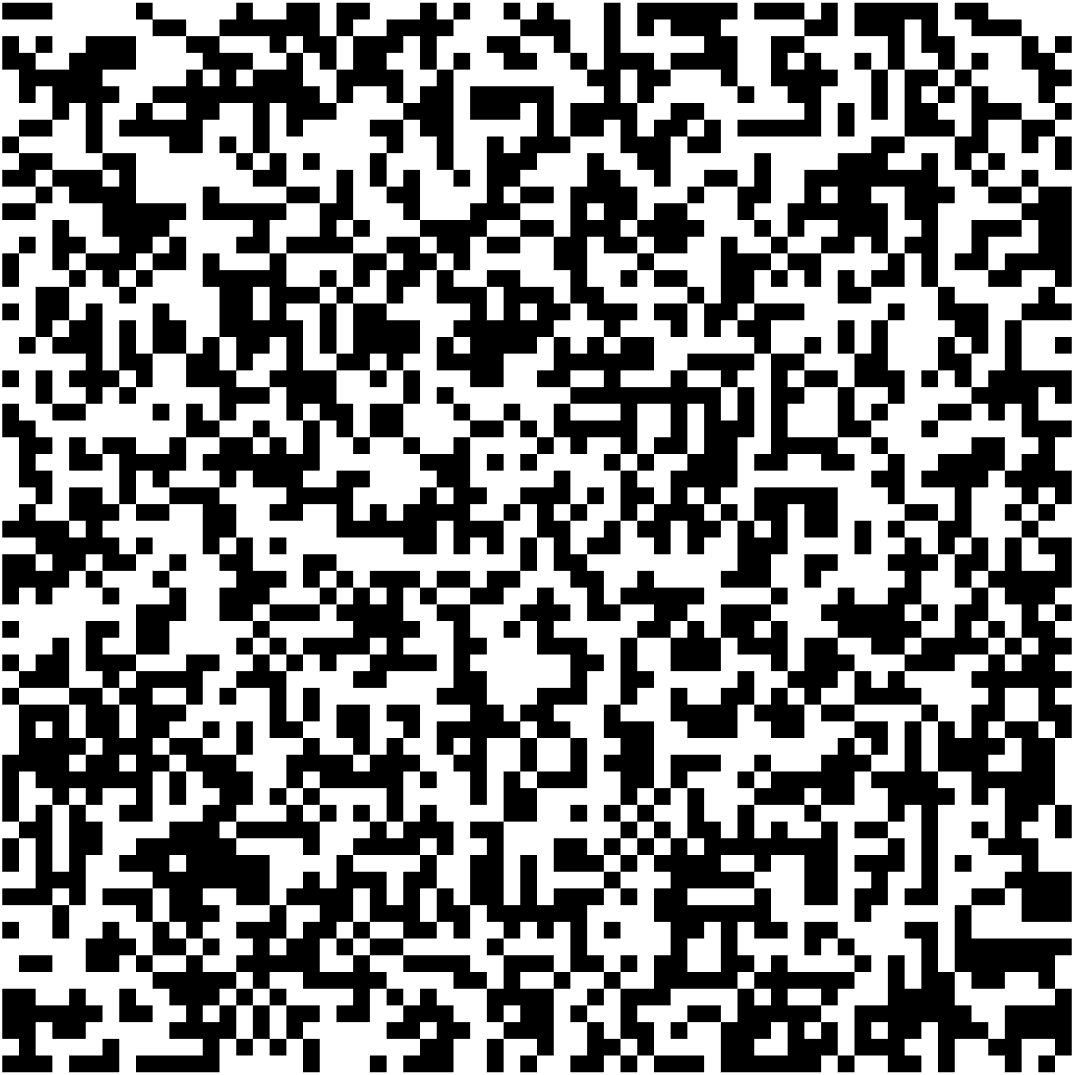}\quad
\includegraphics[width=0.33\linewidth]{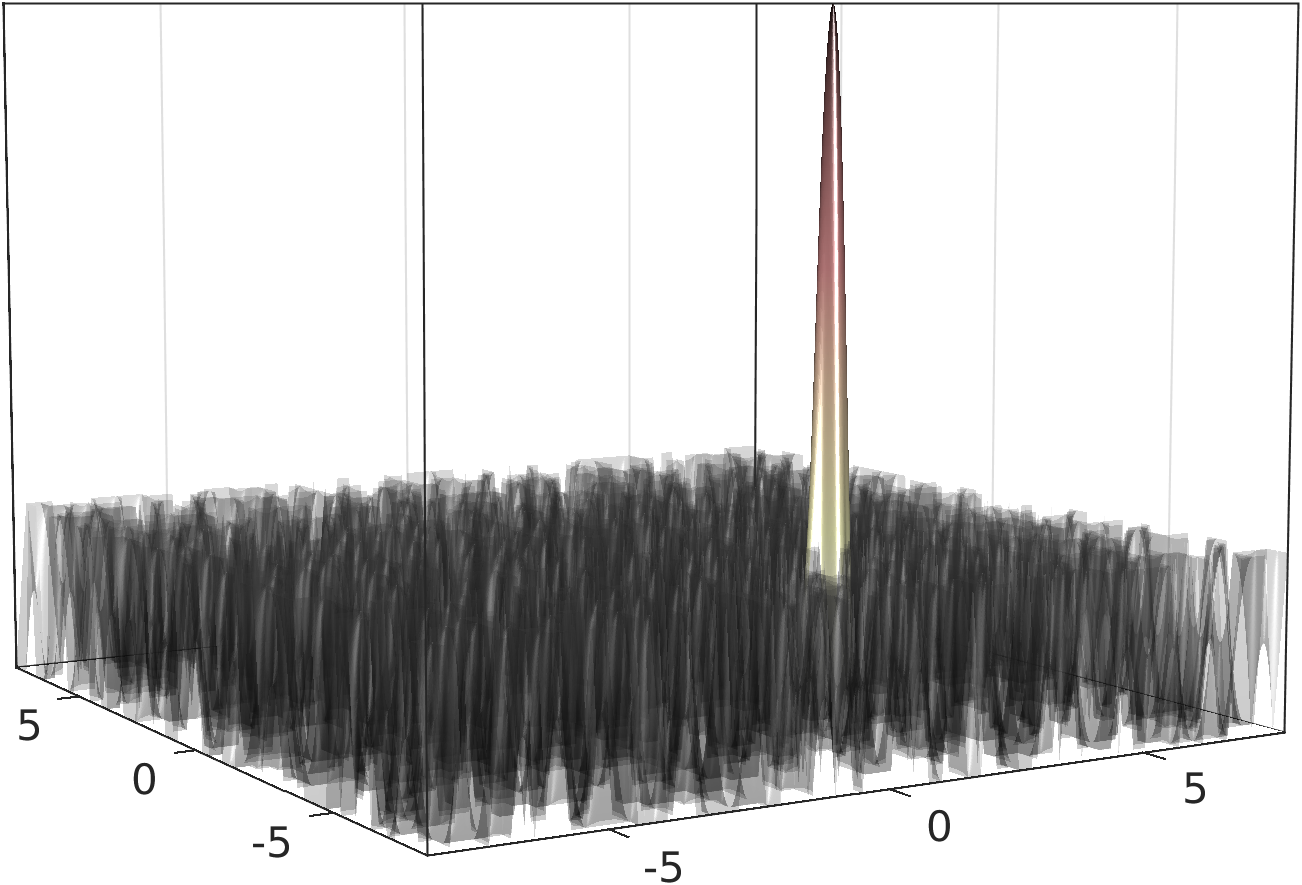}\quad
\includegraphics[width=0.33\linewidth]{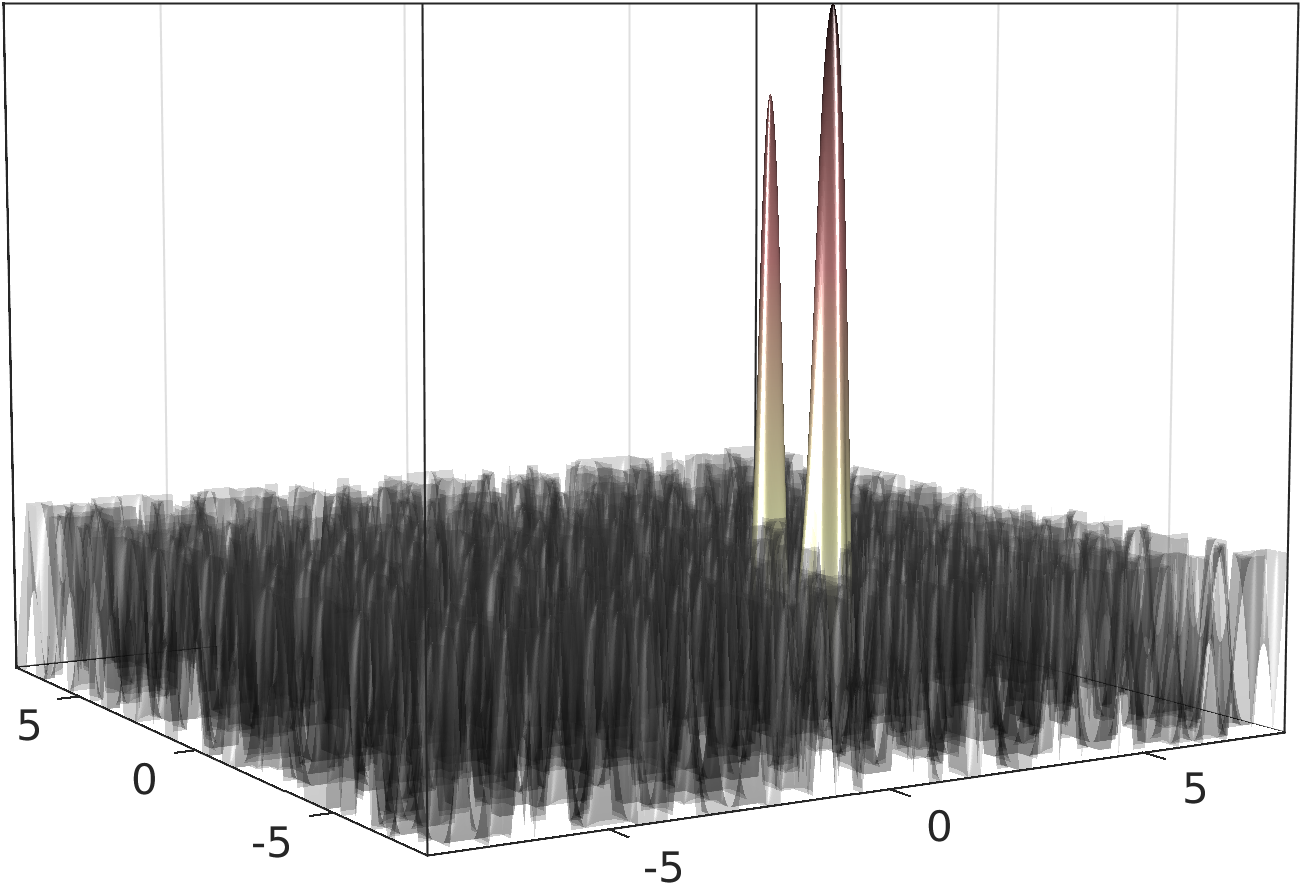}
	\caption{Piecewise constant disorder potential $\vartheta_\text{rand}$ (left, black elements refer to the value~$\varepsilon^{-2}$, white elements refer to the value $0$, $\varepsilon=2^{-6}$) and the corresponding ground states for $\kappa=1$ (middle) and $\kappa =10$ (right).}
	\label{fig:groundstateDisorder}
\end{figure}

Figure~\ref{fig:convergenceDisorder} displays the convergence history of the residuals for~$\kappa=0.1,1,10$. We employed the same strategy as in Section~\ref{sect:numerics:GPEVP:exp1} to generate suitable initial guesses with the residuals of order~$10^{-2}$ used for all methods. The results clearly indicate that the ground state computations with the disorder potential are already challenging for smaller values of $\kappa$. Particularly, the energy-adapted RGD method needs much larger iteration counts, which according to~\cite{Hen23}, may be related to smaller spectral gaps between the first and second eigenvalue. The Riemannian Newton method, on the other hand, still performs well and reaches the prescribed tolerance $10^{-8}$ for the residual in only a few steps in all three examples. For comparison, the SCF iteration converges very fast in the almost linear case but suffers from larger values of~$\kappa$ as the energy-adapted RGD method. Here, again, the higher costs per Newton step are compensated by far by the very small number of needed iteration steps.  
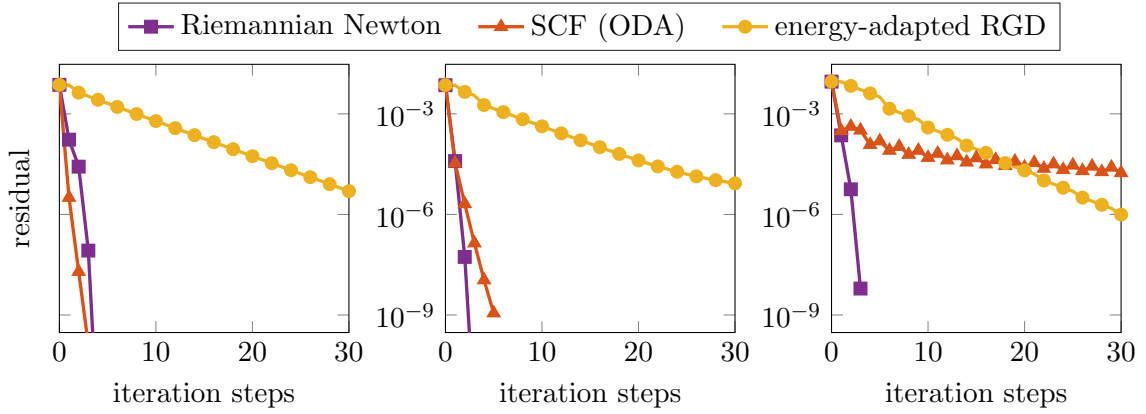
\begin{figure}
%
%
\begin{tikzpicture}

\begin{axis}[%
width=1.5in,
height=1.4in,
at={(0.0in,0.0in)},
scale only axis,
xmin=0,
xmax=30,
xlabel={iteration steps},
ymode=log,
ymin=3e-10,
ymax=0.03,
ylabel={residual},
ylabel near ticks,
yminorticks=true,
yticklabel={\empty},
ytick={1e-3,1e-6,1e-9},
axis background/.style={fill=white},
legend columns = 3,
legend style={legend cell align=left, align=left, at={(1.83,1.05)}, anchor=south, draw=white!15!black}
]
\addplot [color=mycolor4, line width=1.25pt, mark=square*]
  table[row sep=crcr]{%
0	0.0073270004364367\\
1	0.000169657086671074\\
2	2.65996476978382e-05\\
3	8.34641527335287e-08\\
4	3.12934747463332e-13\\
};
\addlegendentry{Riemannian Newton\quad}

\addplot [color=mycolor2, line width=1.25pt, mark=triangle*]
  table[row sep=crcr]{%
0	0.0073270004364367\\
1	3.18525648747122e-06\\
2	1.96327689892877e-08\\
3	1.32399847640338e-10\\
};
\addlegendentry{SCF (ODA)\quad}

\addplot [color=mycolor3, line width=1.25pt, mark=*, mark repeat=2]
  table[row sep=crcr]{%
0	0.0073270004364367\\
1	0.00725456068510416\\
2	0.00428160062896782\\
3	0.00334520693662183\\
4	0.00261754755886466\\
5	0.00205003940601815\\
6	0.00160665507906547\\
7	0.00125981760149577\\
8	0.000988251466584531\\
9	0.000775474368755209\\
10	0.000608670052608787\\
11	0.000477852358189424\\
12	0.00037522302081807\\
13	0.000294686794670846\\
14	0.000231472862912054\\
15	0.000181845913828762\\
16	0.0001428785088297\\
17	0.000112276338280967\\
18	8.82399047272373e-05\\
19	6.93580013672912e-05\\
20	5.4523196219359e-05\\
21	4.28666231518749e-05\\
22	3.37061808388072e-05\\
23	2.65065174475589e-05\\
24	2.08472224333446e-05\\
25	1.63982340693651e-05\\
26	1.29002839057837e-05\\
27	1.01497669217095e-05\\
28	7.98670810142878e-06\\
29	6.28544672437945e-06\\
30	4.94722634918467e-06\\
};
\addlegendentry{energy-adapted RGD} 

\end{axis}
%
%

\begin{axis}[%
width=1.5in,
height=1.4in,
at={(2.0in,0.0in)},
scale only axis,
xmin=0,
xmax=30, 
xlabel={iteration steps},
ymode=log,
ymin=3e-10,
ymax=0.03,
yminorticks=true,
ytick={1e-3,1e-6,1e-9},
axis background/.style={fill=white},
]
\addplot [color=mycolor4, line width=1.25pt, mark=square*]
  table[row sep=crcr]{%
0	0.00714550385519358\\
1	3.90793446101643e-05\\
2	5.39048224151932e-08\\
3	1.11577482377225e-12\\
};

\addplot [color=mycolor2, line width=1.25pt, mark=triangle*]
  table[row sep=crcr]{%
0	0.00714550385519358\\
1	3.37534848620148e-05\\
2	2.05993615465906e-06\\
3	1.39320785677994e-07\\
4	1.10662802620566e-08\\
5	1.13856345601605e-09\\
};

\addplot [color=mycolor3, line width=1.25pt, mark=*, mark repeat=2]
  table[row sep=crcr]{%
0	0.00714550385519358\\
1	0.0070748843293239\\
2	0.00451784328597536\\
3	0.00352927467178377\\
4	0.00183019473800457\\
5	0.0014328272649922\\
6	0.00112233484010931\\
7	0.000879533856971847\\
8	0.000689564311565753\\
9	0.000540873564855092\\
10	0.000424458509714585\\
11	0.000333296389567481\\
12	0.000261901186192119\\
13	0.000205986518537624\\
14	0.000162199285617314\\
15	0.000127916984564057\\
16	0.000101086016273619\\
17	8.00988973362364e-05\\
18	6.36954244801334e-05\\
19	5.08879060567179e-05\\
20	4.09004197126484e-05\\
21	3.31231367999822e-05\\
22	2.70749670498395e-05\\
23	2.23758846954064e-05\\
24	1.87244349088283e-05\\
25	1.58816432129814e-05\\
26	1.36581347352048e-05\\
27	1.19051140024597e-05\\
28	1.05067889494917e-05\\
29	9.37458433531206e-06\\
30	8.44174447894859e-06\\
};

\end{axis}
%
%

\begin{axis}[%
width=1.5in,
height=1.4in,
at={(4.0in,0.0in)},
scale only axis,
xmin=0,
xmax=30,
xlabel={iteration steps},
ymode=log,
ymin=3e-10,
ymax=0.03,
yminorticks=true,
ytick={1e-3,1e-6,1e-9},
axis background/.style={fill=white},
]
\addplot [color=mycolor4, line width=1.25pt, mark=square*]
  table[row sep=crcr]{%
0	0.00913832090896681\\
1	0.000227703732539956\\
2	5.61860863357299e-06\\
3	6.14789161432585e-09\\
};

\addplot [color=mycolor2, line width=1.25pt, mark=triangle*]
  table[row sep=crcr]{%
0	0.00913832090896681\\
1	0.000312259180617435\\
2	0.000418021849975676\\
3	0.000322933191197318\\
4	0.000120374504373254\\
5	0.000154497424600003\\
6	8.04702614251825e-05\\
7	0.000105052198793906\\
8	6.10160927324236e-05\\
9	8.06018097232719e-05\\
10	4.93718621074455e-05\\
11	6.56967266980879e-05\\
12	4.15609022629068e-05\\
13	5.55761074793653e-05\\
14	3.59341395431541e-05\\
15	4.82236619901366e-05\\
16	3.1677145961146e-05\\
17	4.26263615311121e-05\\
18	2.83386685275971e-05\\
19	3.82157081501467e-05\\
20	2.56474509868948e-05\\
21	3.4646625795731e-05\\
22	2.34301790804248e-05\\
23	3.1696938810258e-05\\
24	2.15707269863429e-05\\
25	2.92168576160835e-05\\
26	1.99882613442542e-05\\
27	2.71015806479802e-05\\
28	1.862472384919e-05\\
29	2.52755092792944e-05\\
30	1.74373057325557e-05\\
};

\addplot [color=mycolor3, line width=1.25pt, mark=*, mark repeat=2]
  table[row sep=crcr]{%
0	0.00913832090896681\\
1	0.00904902849826815\\
2	0.00681559234123386\\
3	0.00526422913851671\\
4	0.00408397609718599\\
5	0.00317351736324098\\
6	0.00141904173514767\\
7	0.00110426025300831\\
8	0.000859683472180731\\
9	0.000669508251371807\\
10	0.000390059685034368\\
11	0.000303936816920537\\
12	0.000236887526685005\\
13	0.000184670981289435\\
14	0.000113237006773309\\
15	8.83115846555946e-05\\
16	6.8886072237149e-05\\
17	5.37434148769119e-05\\
18	3.38442092571981e-05\\
19	2.64134541533001e-05\\
20	2.06176620806963e-05\\
21	1.60962362643005e-05\\
22	1.0268271180763e-05\\
23	8.01889562338636e-06\\
24	6.26323302975174e-06\\
25	4.89268885807892e-06\\
26	3.17405490834396e-06\\
27	2.48020239253258e-06\\
28	1.93830689698531e-06\\
29	1.51502417245809e-06\\
30	9.88697852498068e-07\\
};

\end{axis}
\end{tikzpicture}%
	\caption{Convergence history of the residuals for the ground state in a~disorder potential for $\kappa=0.1,1,10$ (from left to right).}
	\label{fig:convergenceDisorder}
\end{figure}
%
\subsection{Kohn--Sham model}\label{sect:numerics:KS}
For the Kohn--Sham energy functional $\calE_{\rm KS}$ introduced in~\eqref{eq:energyKS}, we have 
\begin{align*}
	\langle \calA_\phibf\,\vbf, \wbf\rangle
	& = \int_\Omega \tr\bigl((\nabla \vbf)^T\nabla\wbf\bigr) \dx 
	+ 2 \int_\Omega  \vartheta_\text{ion}\, \vbf\cdot\wbf \dx \\
	& \qquad + 2\int_\Omega \!\Big(\!\int_{\Omega} \frac{\rho(\phibf(y))}{\|x-y\|}\, \dy \Big)\, \vbf\cdot\wbf \dx
	+ 2 \int_\Omega \mu_{\rm xc}(\rho(\phibf))\, \vbf\cdot\wbf \dx
\end{align*}
with $\mu_{\rm xc}(\rho) = \frac{\rm d}{{\rm d}\rho} \big(\rho\,\epsilon_{\rm xc}(\rho)\big)$. Moreover, the operator $\calB_\phibf$ has the form 
\begin{align*}
	\langle\calB_\phibf\,\vbf,\wbf\rangle
	= 4 \int_{\Omega}\!\Big(\!\int_{\Omega} \frac{\phibf\cdot\vbf}{\|x-y\|} \dy \Big) \,\phibf\cdot\wbf \dx 
	+ 4 \int_{\Omega} \zeta_{\rm xc}(\rho(\phibf)) (\phibf\cdot\vbf)\, (\phibf\cdot \wbf) \dx, 
\end{align*} 
where $\zeta_{\rm xc}(\rho)=\frac{{\rm d}}{{\rm d}\rho} \mu_{\rm xc}(\rho)$. 
The exchange-correlation function $\epsilon_\text{xc}(\rho)$ can additively be decomposed as \mbox{$\epsilon_\text{xc}(\rho) =\epsilon_\text{x}(\rho) + \epsilon_\text{c}(\rho)$}, where the exchange component $\epsilon_\text{x}(\rho)$ has the particular analytical expression  $\epsilon_\text{x}(\rho)=-\frac{3}{4}\big(\frac{3}{\pi}\rho\big)^{1/3}$ and the  correlation component $\epsilon_\text{c}(\rho)$ is usually unknown, but can be fitted by using quantum Monte-Carlo data~\cite{PerW92}. 
For the numerical experiments, we use the MATLAB toolbox KSSOLV~\cite{YanMLW09}, in which the correlation component is implemented as 
\begin{equation*}
	\epsilon_\text{c}(\rho) = \left\{\begin{array}{ll} 
		a_1 + a_2\,r(\rho) + \big(a_3 + a_4\,r(\rho)\big) \ln(r(\rho)), &\text{if } r(\rho)<1, \\[0.3em]
		b_1^{-1} \big( 1+b_2\,\sqrt{r(\rho)}+b_3\,r(\rho) \big), 
	 	&\text{if } r(\rho)\geq 1,\end{array}\right.
\end{equation*}
where $r(\rho)=\big(\frac{4\pi}{3}\rho\big)^{-1/3}$ is the Wigner-Seitz radius, and $a_j,b_j\in\R$ are fitted constants; see~\cite[App.~C]{PerZ81}. 

For the spatial discretization, we employ the planewave discretization method as used in KSSOLV. With~$n$ denoting the number of degrees of freedom, the matrix $\Phi\in\C^{n\times p}$ contains the coefficients of the approximation of the wave function $\phibf$. Then the discretized Kohn--Sham energy functional is given by 
\[
E(\Phi) 
= \frac{1}{2} \trace\big(\Phi^*(L+2D_{\rm ion})\Phi\big) + 
\frac{1}{2}\,\rho_h(\Phi)^TL^+\rho_h(\Phi) + \rho_h(\Phi)^T\epsilon_{\rm xc}(\rho_h(\Phi)),
\]
where $\Phi^*$ denotes the complex conjugate transpose of $\Phi$, $L\in\C^{n\times n}$ is the discrete Laplace matrix, $L^+\in\C^{n\times n}$ is its pseudoinverse, \mbox{$D_{\rm ion}\in\R^{n\times n}$} is the discretized ionic potential, and~$\rho_h(\Phi) = \diag(\Phi\Phi^*) \in \R^n$ is the discretized electronic charge density. Note that the matrix $L$ is Hermitian and $D_{\rm ion}$ is diagonal. In this setting, the minimization problem 
\[
\min\limits_{\Phi\in\Stiefelpn} E(\Phi)
\]
 on the (compact) Stiefel manifold 
$
\Stiefelpn 
= \big\{ \Phi\in \C^{n\times p}\; :\; \Phi^* \Phi = I_p \big\} 
$
leads to the finite-dimensional nonlinear eigenvector problem 
\begin{equation*}
	\arraycolsep=2pt
	\begin{array}{rcl}
		A(\Phi)\Phi-\Phi\, \Lambda & = & 0, \\
		\Phi^*\Phi -I_p& = & 0, 
	\end{array}
\end{equation*}
where the discrete Kohn--Sham operator is given by
\[
A(\Phi) 
= L + 2\, D_{\rm ion} + 2\Diag\big(L^+\rho_h(\Phi) + \mu_{\rm xc}(\rho_h(\Phi))\big).
\]
Further, the Riemannian gradient of $E(\Phi)$ becomes
\begin{equation}\label{eq:gradEdiscr}
\grad E(\Phi)
= (I-\Phi\Phi^*)A(\Phi)\Phi
= A(\Phi)\Phi-\Phi\big(\Phi^*A(\Phi)\Phi\big).
\end{equation}
In the Riemannian Newton method on the Stiefel manifold~$\Stiefelpn$, we need to solve the equation 
\begin{equation}\label{eq:NewtonEqSd}
	P_{\Phi}\big(A(\Phi)\Psi+B(\Phi,\Psi)-\Psi\Phi^*A(\Phi)\Phi\big)
	= -(I-\Phi\Phi^*) A(\Phi)\Phi 
\end{equation}
for $\Psi$ belonging to the tangent space~$T_\Phi\, \Stiefelpn$. Therein, 
\[
	P_{\Phi}(Y)
	= Y - \tfrac12\, \Phi \big(\Phi^* Y + Y^* \Phi\big)  
\]
is the orthogonal projector onto~$T_\Phi\,\Stiefelpn$ and 
\[
	B(\Phi,\Psi) 
	= 2\Diag\big((L^++\Diag(\zeta_{\rm xc}(\rho_h(\Phi))))\diag(\Phi\Psi^*+\Psi\Phi^*)\big)\Phi 
\]
is the discretization of the operator $\calB_\phibf$. On the Grassmann manifold $\Grasspn=\Stiefelpn/\OrthGr$, the Newton equation takes the form 
\begin{equation}\label{eq:NewtonEqGd}
	(I-\Phi\Phi^*) \big(A(\Phi)\Psi_{\Phi}^{\rm h} + B(\Phi,\Psi_{\Phi}^{\rm h})-\Psi_{\Phi}^{\rm h}\Phi^*A(\Phi)\Phi\big) 
	= -(I-\Phi\Phi^*) A(\Phi)\Phi
\end{equation}
for $\Psi_\Phi^{\rm h}\in\calH_{\Phi} = \{\Psi\in T_\Phi\, \Stiefelpn\; :\; \Phi^*\Psi=0\}$. 

In our experiments, we compare the calculation of the ground state by using the SCF iteration with the Anderson charge mixing scheme \cite{YanMLW09}, the energy-adaptive RDG with non-monotone step size control \cite{AltPS21}, and the Riemannian Newton methods on the Stiefel manifold (RNS) and on the Grassmann manifold (RNG). In all these methods, we choose the same initial guess for the wave function by performing one SCF step with a~randomly generated starting point and stop the iterations once the Frobenius norm of the residual 
$R(\Phi_k) = A(\Phi_k)\Phi_k-\Phi_k\Lambda_k$
 with $\Lambda_k=\Phi_k^TA(\Phi_k)\Phi_k$ is smaller than the tolerance~$10^{-8}$. Note that due to \eqref{eq:gradEdiscr}, $\|R(\Phi_k)\|_F=\|\grad E(\Phi_k)\|_F$, i.e., the norms of the residuals provide the information on the size of the Riemannian gradients.
In both Newton methods and the energy-adaptive RGD method, we use the qR decomposition based retractions. The reference minimal energy~$E_{\min}$ is computed by the RNG method with the tolerance~$10^{-10}$. 

All algorithms are performed in an {\em inexact manner}, i.e., the occurring linear systems are only solved up to a certain tolerance. In RNG, for instance, we follow Algorithm~\ref{alg:infdim:RiemNewtonGr}  using the MATLAB built-in function {\tt minres} as a~linear system solver with the adaptive tolerance $\min(1/k, 10^{-3} \|\grad E(\Phi_{k-1})\|_F)$ and the maximal number of inner iterations~$\ell_{\max}=15$. The remaining parameters are chosen as~\mbox{$\eta=10^{-8}$}, $\delta=0.5$, and~$\sigma=10^{-4}$. In RNS, we proceed similarly, with the only difference that instead of \eqref{eq:NewtonEqGd} we solve the Newton equation of the form~\eqref{eq:NewtonEqSd}. For solving the linear eigenvalue problems in SCF, we employ the KSSOLV built-in LOBPCG algorithm for the pentacene model in Section~\ref{sect:numerics:KS:exp1} and the MATLAB built-in function {\tt eigs} for the graphene model in Section~\ref{sect:numerics:KS:exp2}. Switching to another eigenvalue solver is necessary due to the ill-conditioning in LOBPCG for the latter example. In both cases, the tolerance for the  inner iterations is set to be $\min(10^{-3}, 10^{-3} \|\grad E(\Phi_{k-1})\|_F)$. For linear system solvers, we use the kinetic energy preconditioner, which provides, especially for the energy-adaptive RGD method, better numerical results than the KSSOLV built-in Teter--Payne--Allan preconditioner. 
	
Finally, we would like to mention that the Riemannian Newton scheme is again robust in terms of the discretization parameter used in KSSOLV. This means that the number of needed iterations of the Riemannian Newton method to fall below a certain tolerance is not effected by finer discretizations (as long as the number of inner iterations steps is sufficiently large). 	
\subsubsection{Pentacene molecule}\label{sect:numerics:KS:exp1}
In the first numerical experiment, we calculate the ground state for the pentacene molecule ${\rm C}_{22}{\rm H}_{14}$ with $p=51$ electron orbitals. 
A~spatial planewave discretization on a~$80\times\ 55\times 160$ sampling grid gives  the discrete model of dimension $n=44791$. In Figure~\ref{fig:convergencePentacene}, we present the convergence history of the residuals  and the energy reduction during the iterations. One can see that both Newton methods have very similar behavior and converge within $8$ and $9$ iterations, respectively. In comparison, the SCF method and the energy-adaptive RGD method require $18$ and $54$ iterations to converge, respectively. In terms of computing time, all methods perform quite similarly in this expe\-ri\-ment. This can also be seen in Table~\ref{eq:KSresults}, which shows the values of the energy functional, the reached residuals, the number of (outer) iterations, the total number of Hamiltonian evaluations, and the CPU time.
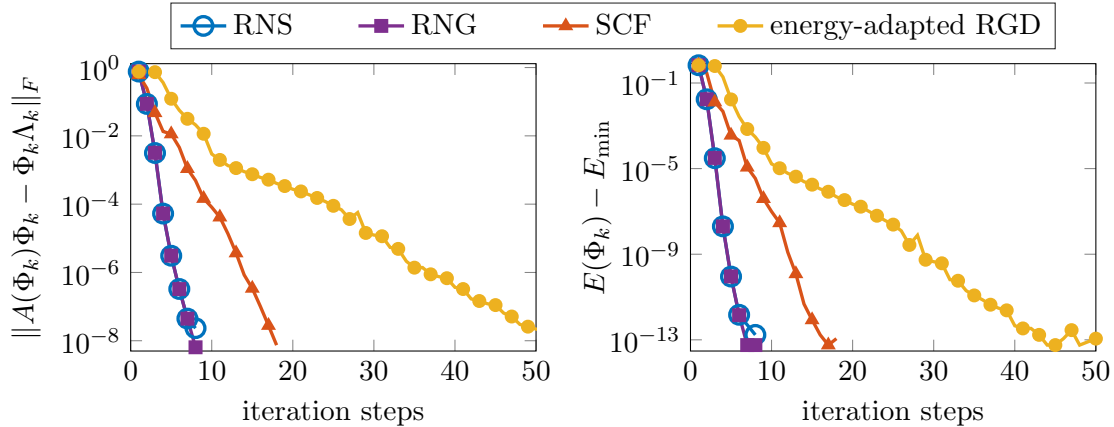
\begin{figure}
%
%
\begin{tikzpicture}

\begin{axis}[%
width=2.1in,
height=1.5in,
at={(0.0in,0.0in)},
scale only axis,
xmin=0.0,
xmax=50,
xlabel={iteration steps},
ymode=log,
ymin=5e-9,
ymax=1.3,
ylabel near ticks,
yminorticks=true,
ylabel={$\|A(\Phi_k)\Phi_k-\Phi_k\Lambda_k\|_F$},
axis background/.style={fill=white},
legend columns = 4,
legend style={legend cell align=left, align=left, at={(1.19,1.05)}, anchor=south, draw=white!15!black}
]

\addplot [color=mycolor1, line width=1.25pt, mark=o, mark size=3.5]
  table[row sep=crcr]{%
1	0.763418551084656\\
2	0.086109452492178\\
3	0.00314558696475516\\
4	5.29308941124361e-05\\
5	3.0860896628072e-06\\
6	3.27923780955634e-07\\
7	4.44964606870635e-08\\
8	2.33325017267211e-08\\
};
\addlegendentry{RNS\qquad}

\addplot [color=mycolor4, line width=1.25pt, mark=square*]
  table[row sep=crcr]{%
1	0.763418551084656\\
2	0.0862059875663052\\
3	0.00318587796986847\\
4	5.3201187627545e-05\\
5	3.10173550411105e-06\\
6	3.2672362407243e-07\\
7	4.43053372016611e-08\\
8	6.43083506394177e-09\\
};
\addlegendentry{RNG\qquad}

\addplot [color=mycolor2, line width=1.25pt, mark=triangle*, mark repeat=2]
  table[row sep=crcr]{%
1	0.610734840890338\\
2	0.246603921134293\\
3	0.0473435394424309\\
4	0.0134288603445182\\
5	0.011329399461423\\
6	0.00459244074773144\\
7	0.00109353562736028\\
8	0.000492102056737164\\
9	0.000145854326011998\\
10	8.12626629323937e-05\\
11	4.16014495284152e-05\\
12	1.35778247616138e-05\\
13	3.70901533425626e-06\\
14	8.46108779337772e-07\\
15	3.41251372375299e-07\\
16	1.00939222714745e-07\\
17	2.78887611045877e-08\\
18	7.43570320734485e-09\\
};
\addlegendentry{SCF\qquad}

\addplot [color=mycolor3, line width=1.25pt, mark=*, mark repeat=2]
  table[row sep=crcr]{%
1	0.763418551084656\\
2	0.743429590416211\\
3	0.733716835838145\\
4	0.362948279955931\\
5	0.121881029816756\\
6	0.0560839971033364\\
7	0.0315812392792368\\
8	0.0209575570793224\\
9	0.0115065428982942\\
10	0.00294450070543436\\
11	0.00197365144616014\\
12	0.00134320941015117\\
13	0.00113889610132987\\
14	0.000897028629195143\\
15	0.000751992262891356\\
16	0.000606446805465877\\
17	0.000520716813160518\\
18	0.000405080991275552\\
19	0.000339015577287257\\
20	0.00027324769815556\\
21	0.000235926764449647\\
22	0.000188635409591681\\
23	0.000151845866211585\\
24	0.000119272228642876\\
25	8.93807865344081e-05\\
26	6.82201503264399e-05\\
27	3.67274975533339e-05\\
28	5.66992950347838e-05\\
29	1.4342650438164e-05\\
30	1.24495435045155e-05\\
31	1.14931218487066e-05\\
32	5.41726350094208e-06\\
33	4.90101462254643e-06\\
34	2.12703694364855e-06\\
35	1.38053543944895e-06\\
36	1.39252101943908e-06\\
37	8.816232946425e-07\\
38	7.54385204810657e-07\\
39	6.78382610186459e-07\\
40	3.66392519653515e-07\\
41	3.2677605763088e-07\\
42	1.6799433310879e-07\\
43	1.45783599969543e-07\\
44	1.21239078035452e-07\\
45	1.10759325260135e-07\\
46	6.46601014707111e-08\\
47	5.12210526303722e-08\\
48	3.01703938484146e-08\\
49	2.58557569141732e-08\\
50	2.19411598350612e-08\\
51	2.0026810836857e-08\\
52	1.42931506810338e-08\\
53	1.00342726584645e-08\\
54	6.98433972155356e-09\\
};
\addlegendentry{energy-adapted RGD} 

\end{axis}

%
%

\begin{axis}[%
width=2.1in,
height=1.5in,
at={(2.9in,0.0in)},
scale only axis,
xmin=0.0,
xmax=50,
xlabel={iteration steps},
ymode=log,
ymin=3e-14,
ymax=0.8,
yminorticks=true,
ylabel={$E(\Phi_k)-E_{\min}$},
axis background/.style={fill=white},
]

\addplot [color=mycolor1, line width=1.25pt, mark=o, mark size=3.5]
  table[row sep=crcr]{%
1	0.676372165321368\\
2	0.0173364663545499\\
3	3.06677701473745e-05\\
4	1.99900682673615e-08\\
5	9.03810359886847e-11\\
6	1.47792889038101e-12\\
7	0\\
8	1.70530256582424e-13\\
};

\addplot [color=mycolor4, line width=1.25pt, mark=square*]
  table[row sep=crcr]{%
1	0.676372165321368\\
2	0.0174421437324099\\
3	3.14391504048217e-05\\
4	1.99358964891871e-08\\
5	9.16884346224833e-11\\
6	1.4210854715202e-12\\
7	5.6843418860808e-14\\
8	5.6843418860808e-14\\
};

\addplot [color=mycolor2, line width=1.25pt, mark=triangle*, mark repeat=2]
  table[row sep=crcr]{%
1	0.676372165321368\\
2	0.295081445425353\\
3	0.0128488520533665\\
4	0.00470580606867088\\
5	0.000359595803331558\\
6	0.000205710102932244\\
7	1.15672256697508e-05\\
8	3.46593861877409e-06\\
9	3.9315744970736e-07\\
10	9.84456960395619e-08\\
11	2.91811943498033e-08\\
12	1.2425402928784e-09\\
13	1.20792265079217e-10\\
14	4.71800376544707e-12\\
15	8.5265128291212e-13\\
16	1.70530256582424e-13\\
17	5.6843418860808e-14\\
18	1.13686837721616e-13\\
};

\addplot [color=mycolor3, line width=1.25pt, mark=*, mark repeat=2]
  table[row sep=crcr]{%
1	0.676372165321368\\
2	0.639319620127139\\
3	0.621570930719429\\
4	0.192503820448792\\
5	0.0170848394922132\\
6	0.00268612130491874\\
7	0.00070712883757551\\
8	0.000297583537587798\\
9	9.32459718114842e-05\\
10	1.66376836432391e-05\\
11	1.03955612758e-05\\
12	6.0274628026491e-06\\
13	4.2040724679282e-06\\
14	2.45810679189162e-06\\
15	1.77825774017037e-06\\
16	1.11760761001278e-06\\
17	8.29587293083023e-07\\
18	4.80723826967733e-07\\
19	3.38843108238507e-07\\
20	2.1838616248715e-07\\
21	1.66547181379428e-07\\
22	1.01026216725586e-07\\
23	6.2762524066784e-08\\
24	3.77303308596311e-08\\
25	2.41894895225414e-08\\
26	1.23909558169544e-08\\
27	2.71262479145662e-09\\
28	7.70626229495974e-09\\
29	5.51381162949838e-10\\
30	4.47585080110002e-10\\
31	3.77156084141461e-10\\
32	6.48014975013211e-11\\
33	5.76392267248593e-11\\
34	1.9838353182422e-11\\
35	1.18802745419089e-11\\
36	7.50333128962666e-12\\
37	4.37694325228222e-12\\
38	3.01270119962282e-12\\
39	2.38742359215394e-12\\
40	4.54747350886464e-13\\
41	3.41060513164848e-13\\
42	3.41060513164848e-13\\
43	1.70530256582424e-13\\
44	5.6843418860808e-14\\
45	5.6843418860808e-14\\
46	1.13686837721616e-13\\
47	2.8421709430404e-13\\
48	5.6843418860808e-14\\
49	0\\
50	1.13686837721616e-13\\
51	1.13686837721616e-13\\
52	0\\
53	5.6843418860808e-14\\
54	5.6843418860808e-14\\
};

\end{axis}
\end{tikzpicture}%
	\caption{Convergence history for the pentacene molecule: residuals (left) and energy reduction (right).}
	\label{fig:convergencePentacene}
\end{figure}
\subsubsection{Graphene lattice}\label{sect:numerics:KS:exp2}
As the second model, we consider a~graphene lattice consisting of carbon atoms arranged in $9$ hexagons with $p=67$ electron orbitals.  We use a~$32\times\ 55\times 160$  sampling grid for the wave function and get a discretized model of dimension  $n=12279$. Figure~\ref{fig:convergenceGraphene} presents the evolution of the residuals and errors in the energy. We observe again that both Newton methods converge very fast compared to the energy-adaptive RGD method which needs $64$ iterations to achieve the tolerance $10^{-8}$ for the residual. In contrary, the SCF iteration has difficulties to converge. Also other mixing strategies implemented in KSSOLV do not improve the convergence property of SCF for the graphene model. This behaviour may be explained by a missing spectral gap between the excited and non-excited states. A detailed comparison, including the overall CPU time is part of Table~\ref{eq:KSresults}. In this experiment (with the particular implementation and used hardware), one can say that the computational complexity of the methods follows the rule of thumb 
\[
\text{1 step Newton}
\quad\approx\quad \text{2 steps SCF}
\quad\approx\quad \text{4 steps eaRGD}.
\]
Overall, this example clearly shows the supremacy of the Newton approach for more challenging examples. 
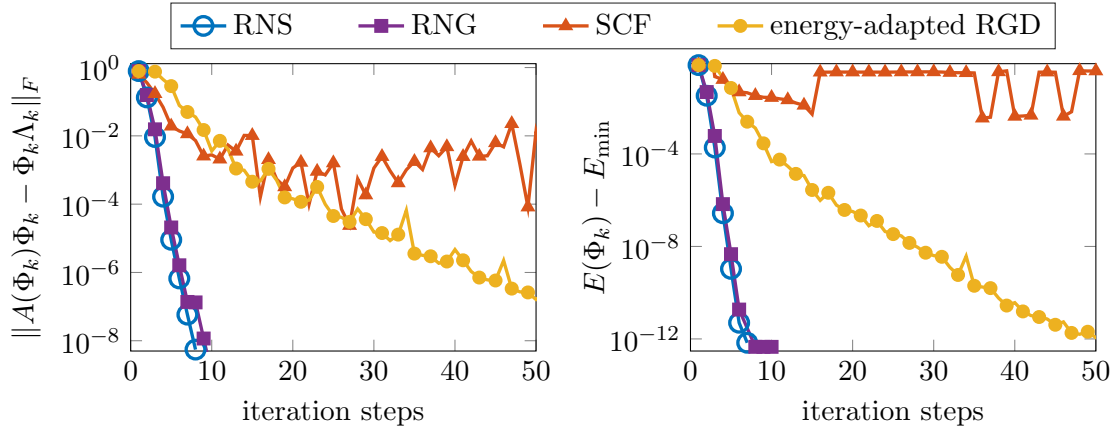
\begin{figure}[h]
%
%
\begin{tikzpicture}

\begin{axis}[%
width=2.1in,
height=1.5in,
at={(0.0in,0.0in)},
scale only axis,
xmin=0.0,
xmax=50,
xlabel={iteration steps},
ymode=log,
ymin=5e-9,
ymax=1.3,
ylabel near ticks,
yminorticks=true,
ylabel={$\|A(\Phi_k)\Phi_k-\Phi_k\Lambda_k\|_F$},
axis background/.style={fill=white},
legend columns = 4,
legend style={legend cell align=left, align=left, at={(1.19,1.05)}, anchor=south, draw=white!15!black}
]

\addplot [color=mycolor1, line width=1.25pt, mark=o, mark size=3.5]
  table[row sep=crcr]{%
1	0.789157381034619\\
2	0.133284597045173\\
3	0.00923354906723674\\
4	0.00016533930333661\\
5	9.04096484648602e-06\\
6	6.66307375727828e-07\\
7	5.80562504590225e-08\\
8	5.50346873449145e-09\\
};
\addlegendentry{RNS\qquad}

\addplot [color=mycolor4, line width=1.25pt, mark=square*]
  table[row sep=crcr]{%
1	0.789157381034619\\
2	0.156513995284413\\
3	0.0155417002275256\\
4	0.000410194589150174\\
5	2.07622967939638e-05\\
6	1.62195933412029e-06\\
7	1.36891933550809e-07\\
8	1.32691908897699e-07\\
9	1.1623038178765e-08\\
10	1.05689156460527e-09\\
};
\addlegendentry{RNG\qquad}

\addplot [color=mycolor2, line width=1.25pt, mark=triangle*, mark repeat=2]
  table[row sep=crcr]{%
1	0.631325904827701\\
2	0.360998671664065\\
3	0.171444189829168\\
4	0.0688779881875539\\
5	0.0191546892669277\\
6	0.0138942779434688\\
7	0.011418762748338\\
8	0.0072761176552412\\
9	0.00254693048543736\\
10	0.00277558575233233\\
11	0.00207718459836053\\
12	0.00538366417075416\\
13	0.00353154953584734\\
14	0.0103053766869891\\
15	0.0101675234083513\\
16	0.000163439817872801\\
17	0.00200124987417178\\
18	0.000540590958294681\\
19	0.000320111678755952\\
20	0.00114195824683009\\
21	0.00166937906754917\\
22	9.88274768594877e-05\\
23	0.00090909319130076\\
24	0.000717328597325182\\
25	0.00160782699413149\\
26	7.16586426820869e-05\\
27	2.36528631180991e-05\\
28	0.000551019599210702\\
29	0.000185225862715035\\
30	0.00113882093095604\\
31	0.00246883890813549\\
32	0.000945109463317894\\
33	0.000418554924585881\\
34	0.00101262028866552\\
35	0.00176157425144311\\
36	0.00103233031687209\\
37	0.00441995505195567\\
38	0.00260950491598576\\
39	0.00477185602044029\\
40	0.000355872794675343\\
41	0.00246155298146531\\
42	0.00707061727059461\\
43	0.00251535099469647\\
44	0.00260349703515596\\
45	0.00626541548781463\\
46	0.00476034776487642\\
47	0.0226779088596933\\
48	0.00150384369771238\\
49	8.20518779812403e-05\\
50	0.0135985229522824\\
51	0.00192619322402335\\
52	0.00735571820743287\\
53	0.00422205833337833\\
54	0.000422490072689081\\
55	0.00069066759402396\\
56	9.8945078155408e-05\\
57	0.000149403892578952\\
58	0.000923916398907472\\
59	0.000346598423828861\\
60	0.00487246254741944\\
61	0.000196303014791948\\
62	0.000183020559095425\\
63	0.00113684140003415\\
64	0.00229411380450612\\
65	0.0017911955921261\\
66	0.000443596908243521\\
67	0.00074611635985785\\
68	0.000979481340225392\\
69	0.000347275060116863\\
70	0.000653551563282048\\
71	0.000662725674047168\\
72	9.91825753120591e-05\\
73	0.00123155887509923\\
74	0.000102045378132028\\
75	0.000121165760295772\\
76	0.000755475825519466\\
77	0.00206005861133624\\
78	0.0026901860204304\\
79	0.000220306908937811\\
80	0.00121030140201641\\
81	0.000927412034775065\\
82	0.00193031788494429\\
83	4.44228742916574e-05\\
84	0.000518095824083885\\
85	0.000344063144498025\\
86	0.000410781940872996\\
87	0.00222499920617796\\
88	0.000260745153227465\\
89	0.000555640742198015\\
90	0.00398983520415842\\
91	0.000336128018220295\\
92	5.93344784626368e-05\\
93	0.0022319319652705\\
94	0.000701300871718884\\
95	0.00281776010322981\\
96	0.00101340636270913\\
97	0.000836658541195903\\
98	0.00026187410429469\\
99	0.000333009911307658\\
100	0.000486571299678415\\
};
\addlegendentry{SCF\qquad}

\addplot [color=mycolor3, line width=1.25pt, mark=*, mark repeat=2]
  table[row sep=crcr]{%
1	0.789157381034619\\
2	0.769382595432285\\
3	0.759530556346563\\
4	0.448146561805064\\
5	0.286446230279962\\
6	0.0777789941944918\\
7	0.0499906841269686\\
8	0.0355530640916504\\
9	0.0148409898723399\\
10	0.00359745720722405\\
11	0.00713749821661088\\
12	0.00412208591894115\\
13	0.00110435989140408\\
14	0.000982851645106313\\
15	0.000454119340116512\\
16	0.000475493049396816\\
17	0.0010789287146293\\
18	0.000503594070556498\\
19	0.000157285841650716\\
20	0.000143545211931814\\
21	0.000117048121018357\\
22	0.000130536017534219\\
23	0.000320057584811724\\
24	9.90803158536712e-05\\
25	4.49858372903834e-05\\
26	4.11148432969137e-05\\
27	3.02658890999706e-05\\
28	7.24884069541188e-05\\
29	3.63668712111011e-05\\
30	1.5229630624881e-05\\
31	1.42594556588984e-05\\
32	8.56435345380458e-06\\
33	1.28166069503673e-05\\
34	6.57645968500905e-05\\
35	3.56002253334299e-06\\
36	3.17530408847339e-06\\
37	2.97555837433178e-06\\
38	1.86091316244904e-06\\
39	2.08867498269702e-06\\
40	4.68069723846984e-06\\
41	2.26786435280508e-06\\
42	1.06960536933268e-06\\
43	7.08450879135005e-07\\
44	6.06865123644374e-07\\
45	5.81029703816048e-07\\
46	2.22967777310854e-06\\
47	3.36322501197875e-07\\
48	2.70746259199353e-07\\
49	2.59876226291712e-07\\
50	1.61032032577731e-07\\
51	2.38997826647913e-07\\
52	1.07973885790536e-06\\
53	1.92461541407434e-07\\
54	8.30487038513595e-08\\
55	6.19774869205472e-08\\
56	7.5642667026465e-08\\
57	4.55942886617535e-08\\
58	3.22142555304986e-08\\
59	2.67670779509939e-08\\
60	4.73215753527752e-08\\
61	4.16083625708835e-08\\
62	1.50345372951208e-08\\
63	1.3447425841797e-08\\
64	9.46451172940898e-09\\
};
\addlegendentry{energy-adapted RGD} 

\end{axis}

%
%

\begin{axis}[%
width=2.1in,
height=1.5in,
at={(2.9in,0.0in)},
scale only axis,
xmin=0.0,
xmax=50,
xlabel={iteration steps},
ymode=log,
ymin=3e-13,
ymax=0.8,
yminorticks=true,
ylabel={$E(\Phi_k)-E_{\min}$},
axis background/.style={fill=white},
]

\addplot [color=mycolor1, line width=1.25pt, mark=o, mark size=3.5]
  table[row sep=crcr]{%
1	0.693755596634446\\
2	0.0328057403987714\\
3	0.000188660219919257\\
4	2.71953922492685e-07\\
5	1.06683728517964e-09\\
6	5.00222085975111e-12\\
7	6.82121026329696e-13\\
};

\addplot [color=mycolor4, line width=1.25pt, mark=square*]
  table[row sep=crcr]{%
1	0.693755596634446\\
2	0.0473862898866173\\
3	0.000597627102024489\\
4	6.75269120620214e-07\\
5	4.62569005321711e-09\\
6	1.8644641386345e-11\\
7	0\\
8	4.54747350886464e-13\\
9	4.54747350886464e-13\\
10	4.54747350886464e-13\\
};

\addplot [color=mycolor2, line width=1.25pt, mark=triangle*, mark repeat=2]
  table[row sep=crcr]{%
1	0.693755596634446\\
2	82.990086002523\\
3	0.221089141899483\\
4	0.164385217669633\\
5	0.0608299793339029\\
6	0.0477350500734701\\
7	0.0441447208081627\\
8	0.0331000166625017\\
9	0.0281204704322136\\
10	0.0274185969899463\\
11	0.0241505826502362\\
12	0.0217795514990939\\
13	0.0159180472667231\\
14	0.0126719215752473\\
15	0.00523346323780061\\
16	0.342614495149292\\
17	0.344025713996643\\
18	0.347509429727552\\
19	0.34627587041382\\
20	0.346864258624009\\
21	0.345218424990662\\
22	0.345035747961674\\
23	0.344899206084165\\
24	0.343390712555447\\
25	0.345877061923602\\
26	0.350794978844078\\
27	0.351121418987304\\
28	0.351278551265978\\
29	0.35476682360968\\
30	0.353594596607763\\
31	0.346923247497898\\
32	0.33180901929768\\
33	0.324997223612399\\
34	0.324213488391024\\
35	0.317787779303217\\
36	0.00352464510433492\\
37	0.00386311838042275\\
38	0.354117586439543\\
39	0.390979005027475\\
40	0.00420657873337404\\
41	0.00436728104409667\\
42	0.00467790676452751\\
43	0.343590636117597\\
44	0.346401229724279\\
45	0.340692425236966\\
46	0.00423898589156124\\
47	0.00664797762510716\\
48	0.394392988201844\\
49	0.385442540917666\\
50	0.385945822479925\\
51	0.319664670466864\\
52	0.324319945110346\\
53	0.378052381510088\\
54	0.392787517789657\\
55	0.3951938337093\\
56	0.393110763679942\\
57	0.394219088668706\\
58	0.395897149934171\\
59	0.402644040020959\\
60	0.399542004726072\\
61	0.351340106387397\\
62	0.348657638772465\\
63	0.346957080163065\\
64	0.350576652182554\\
65	0.365054228833287\\
66	0.365498714126261\\
67	0.357106822804553\\
68	0.364690534011743\\
69	0.355030787983878\\
70	0.35082737070752\\
71	0.360031809294242\\
72	0.358741737526316\\
73	0.35691905044132\\
74	0.344197641922392\\
75	0.34324775867276\\
76	0.344785962213336\\
77	0.352359105422238\\
78	0.339927090335095\\
79	0.366589723682864\\
80	0.366461777402719\\
81	0.372120439402352\\
82	0.383246174153555\\
83	0.397318960881421\\
84	0.397977709679935\\
85	0.403950046511909\\
86	0.401341606183905\\
87	0.398418567113822\\
88	0.413388005945762\\
89	0.413877977234051\\
90	0.413229984702411\\
91	0.441748045678878\\
92	0.444155783603719\\
93	0.444294919768708\\
94	0.459637272042528\\
95	0.454876520767129\\
96	0.468700688651779\\
97	0.472833019797918\\
98	0.467140456885318\\
99	0.469280523674342\\
100	0.467528981003625\\
};

\addplot [color=mycolor3, line width=1.25pt, mark=*, mark repeat=2]
  table[row sep=crcr]{%
1	0.693755596634446\\
2	0.658503097998846\\
3	0.641259182070144\\
4	0.196322349592947\\
5	0.0710942389246156\\
6	0.0060411813265091\\
7	0.00247577165464463\\
8	0.00130872818249372\\
9	0.000286994674524976\\
10	4.65813791379333e-05\\
11	5.67664851587324e-05\\
12	2.93930079351412e-05\\
13	1.37279396312806e-05\\
14	1.16042569970887e-05\\
15	2.76040827884572e-06\\
16	9.80698814601055e-07\\
17	2.08838469006878e-06\\
18	6.11355062574148e-07\\
19	3.78347294827108e-07\\
20	3.23746235153521e-07\\
21	2.17096157939523e-07\\
22	7.84964413469424e-08\\
23	1.26440227177227e-07\\
24	4.42398686573142e-08\\
25	3.41226495947922e-08\\
26	2.8652038963628e-08\\
27	1.41776581585873e-08\\
28	9.25592757994309e-09\\
29	5.34623723069672e-09\\
30	4.01519173465203e-09\\
31	3.55885276803747e-09\\
32	1.24236976262182e-09\\
33	5.82986103836447e-10\\
34	3.77167452825233e-09\\
35	1.98951966012828e-10\\
36	1.75759851117618e-10\\
37	1.5756995708216e-10\\
38	5.84350345889106e-11\\
39	2.77395884040743e-11\\
40	3.70619090972468e-11\\
41	1.54614099301398e-11\\
42	1.02318153949454e-11\\
43	8.86757334228605e-12\\
44	7.27595761418343e-12\\
45	4.09272615797818e-12\\
46	5.45696821063757e-12\\
47	1.81898940354586e-12\\
48	1.59161572810262e-12\\
49	2.04636307898909e-12\\
50	1.13686837721616e-12\\
51	9.09494701772928e-13\\
52	1.13686837721616e-12\\
53	1.13686837721616e-12\\
54	1.13686837721616e-12\\
55	1.13686837721616e-12\\
56	1.13686837721616e-12\\
57	4.54747350886464e-13\\
58	2.27373675443232e-13\\
};

\end{axis}
\end{tikzpicture}%
	\caption{Convergence history for the graphene lattice: residuals (left) and energy reduction (right).}
	\label{fig:convergenceGraphene}
\end{figure}
\begin{table}[th]
	\caption{Numerical results for the pentacene and the graphene models. }
	\label{eq:KSresults}
	\centering
	\begin{tabular}{r||cc|cr|r}		
		method & energy & residual & \#  iter & \# Ham.~eval. & CPU time [$s$] \enskip\; \\[0.2em]
		\hline 
		\hline  
		& \multicolumn{5}{l}{pentacene molecule\hfill $n=44791$, $p=51$} \\[0.1em] \hline
		RNS   & $-1.3189\,e\!+\!2\;$ & $3.2731\,e\!-\!9$ & $\enskip 10$ & $208$ & $1880.64$\\ 
		RNG   & $-1.3189\,e\!+\!2\;$ & $6.4308\,e\!-\!9$ & $\quad 8$    & $172$ & $1572.17$\\ 
		SCF   & $-1.3189\,e\!+\!2\;$ & $7.4357\,e\!-\!9$ & $\enskip 18$ & $276$ & $1709.78$\\ 
		eaRGD & $-1.3189\,e\!+\!2\;$ & $6.9843\,e\!-\!9$ & $\enskip 54$ & $369$ & $1739.08$\\[0.4em] 
		\hline \hline
		& \multicolumn{5}{l}{graphene lattice\hfill $n=12279$, $p=67$} \\[0.1em] \hline
		RNS   & $-1.7360\,e\!+\!2\;$ & $\;5.5035\,e\!-\!9\;$ & $\quad 8$    & $657$   &  $568.69$\\  
		RNG   & $-1.7360\,e\!+\!2\;$ & $\;1.0569\,e\!-\!9\;$ & $\enskip 10$ & $693$   &  $741.55$\\  
		SCF   & $-1.7312\,e\!+\!2\;$ & $\;4.8657\,e\!-\!4\;$ & $100$        & $69616$ &  $3422.67$\\ 
		eaRGD & $-1.7360\,e\!+\!2\;$ & $\;9.4645\,e\!-\!9\;$ & $\enskip 64$ & $914$   &  $991.22$    
	\end{tabular}
\end{table}
%
%
\section{Conclusion}
In this paper, we have derived Riemannian Newton methods on the infinite-dimensional Stiefel and Grassmann manifolds for Kohn--Sham type energy minimization problems. Starting from an energy functional, we present a unified approach for applications in computational physics (e.g., the Gross--Pitaevskii eigenvalue problem) and computational chemistry (e.g., the Kohn--Sham model). 
The remarkable gain in computational efficiency of the Riemannian Newton methods compared to the so far more popular methods such as SCF and gradient descent methods is demonstrated by a series of numerical experiments. 
%
%
\bibliographystyle{alpha}
\bibliography{refKS}
\end{document}